\documentclass[10pt, twoside]{scrartcl}
\usepackage{etex}
\usepackage[a4paper, top=88pt, bottom=88pt, left=72pt, right=72pt, headsep=16pt, footskip=28pt]{geometry}
\usepackage{epsfig}
\usepackage{amsmath, amssymb, amsthm, xcolor, tabularx}
\usepackage{bbm}
\usepackage[hyphens]{url}
\usepackage{hyperref}
\hypersetup{
	hidelinks,
	colorlinks,
	linkcolor=blue,
	urlcolor=blue,
	citecolor=blue,
}
\RequirePackage[hyphens]{url}
\usepackage[nameinlink]{cleveref}
\usepackage{enumitem}
\usepackage{mleftright} \mleftright
\renewcommand{\qedsymbol}{$\blacksquare$}
\renewenvironment{proof}[1][\proofname]{\noindent{\bfseries #1.} }{\hfill \qedsymbol \medskip}

\usepackage{scrlayer-scrpage, xhfill}
\automark[section]{section}
\setkomafont{pageheadfoot}{\normalcolor\sffamily}
\setkomafont{pagenumber}{\normalfont\normalsize\sffamily}
\clearpairofpagestyles
\let\oldtitle\title
\renewcommand{\title}[1]{\oldtitle{#1}\newcommand{\theshorttitle}{#1}}
\newcommand{\shorttitle}[1]{\renewcommand{\theshorttitle}{#1}}
\let\oldauthor\author
\renewcommand{\author}[1]{\oldauthor{#1}\newcommand{\theshortauthor}{#1}}
\newcommand{\shortauthor}[1]{\renewcommand{\theshortauthor}{#1}}
\cohead{\xrfill[0.525ex]{0.6pt}~\theshorttitle~\xrfill[0.525ex]{0.6pt}}
\cehead{\xrfill[0.525ex]{0.6pt}~\theshortauthor~\xrfill[0.525ex]{0.6pt}}
\newcommand*{\affilref}[1]{\ref{affiliation#1}}
\newcommand*{\affiliation}[3]{
	\footnotetext[#1]{\label{affiliation#2} #3}
}

\usepackage{natbib}
\setcitestyle{round}
\usepackage{stmaryrd}

\usepackage[lined, ruled]{algorithm2e}

\usepackage[loose]{subfigure}
\usepackage{gensymb}

\crefname{Assumption}{Assumption}{Assumptions}
\crefname{Algorithm}{Algorithm}{Algorithms}
\crefname{Proposition}{Proposition}{Propositions}

\newtheorem{Theorem}{Theorem}[section]

\newtheorem{Proposition}[Theorem]{Proposition}
\theoremstyle{definition}
\newtheorem{Definition}[Theorem]{Definition}

\newtheorem{Assumption}[Theorem]{Assumption}
\newtheorem{Remark}[Theorem]{Remark}

\newcommand*{\arXiv}[1]{\bgroup\color{blue}\href{https://arxiv.org/abs/#1}{arXiv:#1}\egroup}
\newcommand*{\doi}[1]{\href{https://doi.org/#1}{\nolinkurl{doi:#1}}}
\newcommand*{\email}[1]{\bgroup\color{blue}\href{mailto:#1}{#1}\egroup}
\renewcommand*{\url}[1]{\bgroup\color{blue}\href{#1}{#1}\egroup}
\newcommand{\todo}[1]{\bgroup\color{red}\bfseries#1\egroup}

\renewcommand*{\paragraph}[1]{\smallskip\noindent\textbf{\textsf{#1}}\,\,}

\usepackage{mathtools}
\newcommand*{\argmax}{\mathop{\textup{arg\,max}}}
\newcommand*{\argmin}{\mathop{\textup{arg\,min}}}

\newcommand*{\defeq}{\coloneqq}

\newcommand*{\Naturals}{\mathbb{N}}

\renewcommand*{\P}{\mathbb{P}}

\newcommand*{\quark}{\setbox0\hbox{$x$}\hbox to\wd0{\hss$\cdot$\hss}}

\newcommand*{\Reals}{\mathbb{R}}

\newcommand*{\Absval}[1]{\left\vert #1 \right\vert}
\newcommand*{\step}[2]{#1^{(#2)}}

\newcommand*{\Fdag}{F^{\dagger}}
\newcommand*{\Fnum}{G}
\newcommand*{\Fapp}{F}
\newcommand*{\yobs}{y}
\newcommand*{\quality}{q}
\newcommand*{\area}{a}
\newcommand*{\varea}{{\boldsymbol\area}}
\newcommand*{\Nobs}{I}
\newcommand*{\Nstep}{N}
\newcommand*{\ER}{\mathcal{E}}
\newcommand*{\Area}{\mathrm{A}}
\newcommand*{\WA}{\mathrm{WA}}
\newcommand*{\indic}{\mathbbm{1}}
\newcommand*{\demi}{\frac{1}{2}}

\numberwithin{equation}{section}
\numberwithin{figure}{section}
\numberwithin{table}{section}

\newcommand{\revised}[1]{\bgroup\color{blue}#1\egroup}
\newcommand{\bbracket}[1]{\llbracket#1\rrbracket}
\newcommand{\subseq}{\varphi}

\renewcommand{\revised}[1]{#1}

\setlist[itemize]{nosep}
\setlist[enumerate]{nosep}

\begin{document}

\title{Adaptive reconstruction of imperfectly-observed monotone functions, with applications to uncertainty quantification}
\shorttitle{Adaptive reconstruction of imperfectly-observed monotone functions}

\author{%
	L.\ Bonnet\textsuperscript{\affilref{ONERA},\affilref{MSSMat}}
	\and
	J.-L.\ Akian\textsuperscript{\affilref{ONERA}}
	\and
	{\'E}.\ Savin\textsuperscript{\affilref{ONERA}}
	\and
	T.~J.\ Sullivan\textsuperscript{\affilref{Warwick},\affilref{ZIB}}%
}
\shortauthor{L.~Bonnet, J.-L.~Akian, \'E.~Savin, and T.~J.~Sullivan}

\newcommand{\LBsays}[1]{\bgroup\color{orange}{\text{LB says: }#1}\egroup}
\newcommand{\ESsays}[1]{\bgroup\color{green}{\text{\'ES says: }#1}\egroup}
\newcommand{\TJSsays}[1]{\bgroup\color{purple}{\text{TJS says: }#1}\egroup}
\newcommand{\citationneeded}{\textsuperscript{[\bgroup\color{blue}citation needed\egroup]}}

\date{\today}

\maketitle

\affiliation{1}{ONERA}{ONERA, 29 Avenue de la Division Leclerc, 92320 Ch{\^a}tillon, France (\email{luc.bonnet@onera.fr}, \email{jean-luc.akian@onera.fr}, \email{eric.savin@onera.fr})}
\affiliation{2}{MSSMat}{Laboratoire MSSMat - UMR CNRS 8579, CentraleSup{\'e}lec, 8--10 rue Joliot Curie, 91190 Gif sur Yvette, France (\email{luc.bonnet@ens-paris-saclay.fr})}
\affiliation{3}{Warwick}{Mathematics Institute and School of Engineering, University of Warwick, Coventry CV4 7AL, United Kingdom (\email{t.j.sullivan@warwick.ac.uk})}
\affiliation{4}{ZIB}{Zuse Institute Berlin, Takustra{\ss}e 7, 14195 Berlin, Germany (\email{sullivan@zib.de})}

\begin{abstract}
	\paragraph{Abstract:}
	Motivated by the desire to numerically calculate rigorous upper and lower bounds on deviation probabilities over large classes of probability distributions, we present an \revised{adaptive} algorithm for the reconstruction of increasing real-valued functions. While this problem is similar to the classical statistical problem of isotonic regression, the optimisation setting alters several characteristics of the problem and opens natural algorithmic possibilities. We present our algorithm, establish sufficient conditions for convergence of the reconstruction to the ground truth, and apply the method to synthetic test cases and a real-world example of uncertainty quantification for aerodynamic design.

	\paragraph{Keywords:}
	adaptive approximation $\bullet$
	isotonic regression $\bullet$
	optimisation under uncertainty $\bullet$
	uncertainty quantification $\bullet$
	aerodynamic design

	\paragraph{2010 Mathematics Subject Classification:}
	49M30 
	$\bullet$
	62G08 
	$\bullet$
	68T37 
\end{abstract}

\section{Introduction}
\label{sec:introduction}%

This paper considers the problem of adaptively reconstructing a monotonically increasing function $\Fdag$ from imperfect pointwise observations of this function.
In the statistical literature, the problem of estimating a monotone function is commonly known as \emph{isotonic regression}, and it assumed that the observed data consist of noisy pointwise evaluations of $\Fdag$.
However, we consider this problem under assumptions that differ from the standard formulation, and these differences motivate our algorithmic approach to the problem.
To be concrete, our two motivating examples are that
\begin{equation}
	\label{eq:F_dagger_intro_1}
	\Fdag (x) \defeq \P_{\Xi \sim \mu} [ g(\Xi) \leq x ]
\end{equation}
is the cumulative distribution function (CDF) of a known real-valued function $g$ of a random variable $\Xi$ with known distribution $\mu$, or that
\begin{equation}
	\label{eq:F_dagger_intro_2}
	\Fdag (x) \defeq \sup_{(g, \mu) \in \mathcal{A}} \P_{\Xi \sim \mu} [ g(\Xi) \leq x ]
\end{equation}
is the supremum of a family of such CDFs over some class $\mathcal{A}$.
We assume that we have access to a numerical optimisation routine that can, for each $x$ and some given numerical parameters $\quality$ (e.g.\ the number of iterations or other convergence tolerance parameters), produce a \emph{numerical estimate} or \emph{observation} $\Fnum(x,\quality)$ of $\Fdag(x)$; 
furthermore, we assume that $\Fnum(x,\quality) \leq\Fdag(x)$ is always true, i.e.\ the numerical \revised{optimisation} routine always under-estimates the true optimum value, and that the positive error $\Fdag(x) - \Fnum(x,\quality)$ can be controlled to some extent through the choice of the optimisation parameters $\quality$, but remains essentially influenced by randomness in the optimisation algorithm for each $x$.
\revised{The assumption $\Fnum(x,\quality) \leq\Fdag(x)$ is for example coherent with either \Cref{eq:F_dagger_intro_1}, which may be approached by increasing the number of samples (say $\quality$) in a Monte Carlo simulation, or \Cref{eq:F_dagger_intro_2}, which is a supremum over a set that may be explored only partially by the algorithm.}

A single observation $\Fnum(x, q)$ yields some limited information about $\Fdag(x)$;
a key limitation is that one may not even know a priori how accurate $\Fnum(x, q)$ is.
Naturally, one may repeatedly evaluate $\Fnum$ at $x$, perhaps with different values of the optimisation parameters $\quality$, in order to more accurately estimate $\Fdag(x)$.
\revised{However, a key observation is that a \emph{suite} of observations $\Fnum(x_i, \quality_{i})$, $i = 1, \dots, I$, contains much more information than simply estimates of $\Fdag(x_{i})$, $i = 1, \dots, I$, and this information can and must be used.
For example, suppose that the values $(\Fnum(x_i, \quality_{i}))_{i = 1}^{I}$ are not increasing, e.g.\ because}
\[
	\Fnum(x_i, \quality_{i}) > \Fnum(x_{i'}, \quality_{i'})
	\quad
	\text{and}
	\quad
	x_{i} < x_{i'} .
\]
\revised{Such a suite of observations would be inconsistent with the axiomatic requirement that $F^{\dagger}$ is an increasing function.
In particular, while the observation at $x_{i}$ may be relatively good or bad on its own merits, the observation $\Fnum(x_{i'}, \quality_{i'})$ at $x_{i'}$, which violates monotonicity, is in some sense ``useless'' as it gives no better lower bound on $\Fdag(x_{i'})$ than the observation at $x_{i}$ does.
The observation at $x_{i'}$ is thus a good candidate for repetition with more stringent optimisation parameters $\quality$ --- and this is not something that could have been known without comparing it to the rest of the data set.}

The purpose of this article is to leverage this and similar observations to define an algorithm for the reconstruction of the function $\Fdag$, repeating old observations of insufficient quality and introducing new ones as necessary.
The principal parameter in the algorithm is an ``exchange rate'' $\ER$ that quantifies the degree to which the algorithm prefers to have a few high-quality evaluations versus many poor-quality evaluations.
Our approach is slightly different from classical isotonic (or monotonic) regression, which is understood as the least-squares fitting of an increasing function to a set of points in the plane.
The latter problem is uniquely solvable and its solution can be constructed by the pool adjacent violators algorithm (PAVA) extensively studied in \citet{Barlow1972}.
This algorithm consists of exploring the data set from left to right until the monotonicity condition is violated, and replacing the corresponding observations by their average while back-averaging to the left if needed to maintain monotonicity.
Extensions to the PAVA have been developed by \citet{Jan2009} to consider non least-squares loss functions and repeated observations, by \citet{Tibshirani2011} to consider ``nearly-isotonic'' or ``nearly-convex'' fits, and by \citet{Jordan2019} to consider general loss functions and partially ordered data sets.
Useful references on isotonic regression also include \citet{Robertson1988} and \citet{Groeneboom2014}.

The remainder of this paper is structured as follows.
\Cref{sec:setup} presents the problem description and notation, after which the proposed adaptive algorithm for the reconstruction of $\Fdag$ is presented in \Cref{sec:algorithms}.
We demonstrate the convergence properties of the algorithm in \Cref{sec:proof} and study its performance on several analytically tractable test cases in \Cref{sec:tests}.
\Cref{sec:application_ouq} details the application of the algorithm to a challenging problem of the form \Cref{eq:F_dagger_intro_2} drawn from aerodynamic design.
Some closing remarks are given in \Cref{sec:conclusion}.

\section{Notation and problem description}%
\label{sec:setup}%

In the following, the ``ground truth'' response function that we wish to reconstruct is denoted $\Fdag \colon [a,b] \to \Reals$ and has inputs $x \in [a,b] \subset \Reals$.
It is assumed that $\Fdag$ is monotonically increasing \revised{and non-constant on $[a,b]$}.
In contrast, $\Fnum\colon [a,b]\times\Reals_{+} \to \Reals$ denotes the numerical process used to obtain an imperfect pointwise observation $\yobs$ of $\Fdag(x)$ at some point $x\in[a,b]$ for some numerical parameter $\quality\in\Reals_{+}$.
Here, on a heuristic level, $\quality > 0$ stands for the ``quality'' of the noisy evaluation $\Fnum (x, \quality)$.

The main aim of this paper is to show the effectiveness of the proposed algorithm for the adaptive reconstruction of $\Fdag$, which could be continuous or not, from imperfect pointwise observations $\Fnum(x_{i}, \quality_{i})$ of $\Fdag$, where we are free to choose $x_{i + 1}$ and $\quality_{i + 1}$ adaptively based upon $x_{j}$, $\quality_{j}$, and $\Fnum(x_{j}, \quality_{j})$ for $j \leq i$

First, we associate with $\Nobs$ imperfect pointwise observations \revised{$\smash{\{x_i, \yobs_i \defeq \Fnum(x_i,\quality_i)\}_{i=1}^{\Nobs}} \subset [a,b] \times \Reals$}, positive numbers $\{\quality_i\}_{i = 1}^{\Nobs} \subset \Reals_{+}$ which we will call \emph{qualities}.
The quality $\quality_i$ quantifies the confidence we have in the pointwise observation $\yobs_i$ of $\Fdag(x_i)$ using the numerical process $\Fnum(x_i,\quality_i)$.
The higher this value, the greater the confidence.
We divide this quality as the product of two different numbers $c_i$ and $r_i$, $\quality_i=c_i\times r_i$, with the following definitions:
\begin{itemize}
	\item \emph{Consistency} $c_i \in \{0,1\}$\string: This describes the fact that two successive points must be monotonically consistent with respect to each other.
	That is, when one takes two input values $x_2 > x_1$, one should have $\yobs_2 \geq \yobs_1$ as $\yobs$ must be monotonically increasing.
	There is no consistency associated with the very first data point as it does not have any predecessor.
	\item \emph{Reliability} $r_i \in \mathbb{R}_{+}$\string: This describes how confident we are about the numerical value.
	Typically, it will be related to some error estimator if one is available, or the choice of optimisation parameters.
	It is expected that the higher the reliability, the closer the pointwise observation is to the true value, in average.	
\end{itemize}
Typically, if the observation $\yobs_{i+1}=\Fnum(x_{i+1},\quality_{i+1})$ is consistent with regard to the observation $\yobs_i=\Fnum(x_i,\quality_i)$ where $x_{i+1} > x_i$, the quality $\quality_{i+1}$ associated with $\yobs_{i+1}$ will be equal to $\quality_{i+1} = r_{i+1}\in\Reals_{+}^{*}$ since $c_{i+1} = 1$ in this case.
If the value is not consistent, we have $\quality_{i+1} = r_{i+1}\times c_{i+1} = 0$.
Finally, if $x = a$ there is no notion of consistency as there is no point preceding it.
Thereby, the quality associated with this point is only equal to its reliability.

Moreover, we associate to these pointwise observations a notion of area, illustrated in \Cref{fig:possible_func_and_piecewise_constant} and defined as follows.
Consider two consecutive points $x_i$ and $x_{i+1}$ with their respective observations $\yobs_i$ and $\yobs_{i+1}$, the area $\area_i$ for these two points is 
\begin{equation}
	\label{eq:area}
	\area_i = (x_{i+1}-x_i)\times(\yobs_{i+1} - \yobs_i)\,.
\end{equation}
Thus, we can define a vector $\varea= \{\area_i\}_{i=1}^{\Nobs-1}$ which contains all the computed areas for the whole dataset.
In addition, we can assure that if we take two points $x_1$ and $x_2>x_1$ with $\yobs_1 = \Fdag(x_1)$ and $\yobs_2 = \Fdag(x_2)$---namely, the error at these point is equal to zero, the graph of ground truth function $\Fdag$ must lie in the rectangular area spanned by the two points $(x_{1}, \Fdag(x_{1}))$ and $(x_{2}, \Fdag(x_{2}))$.

\begin{figure}[t]%
	\centering%
	\subfigure[Possible ground truth functions between two consecutive points $x_1$ and $x_2$.
	The ground truth function must lie in the area formed by these two points.]{\includegraphics[width=0.45\linewidth]{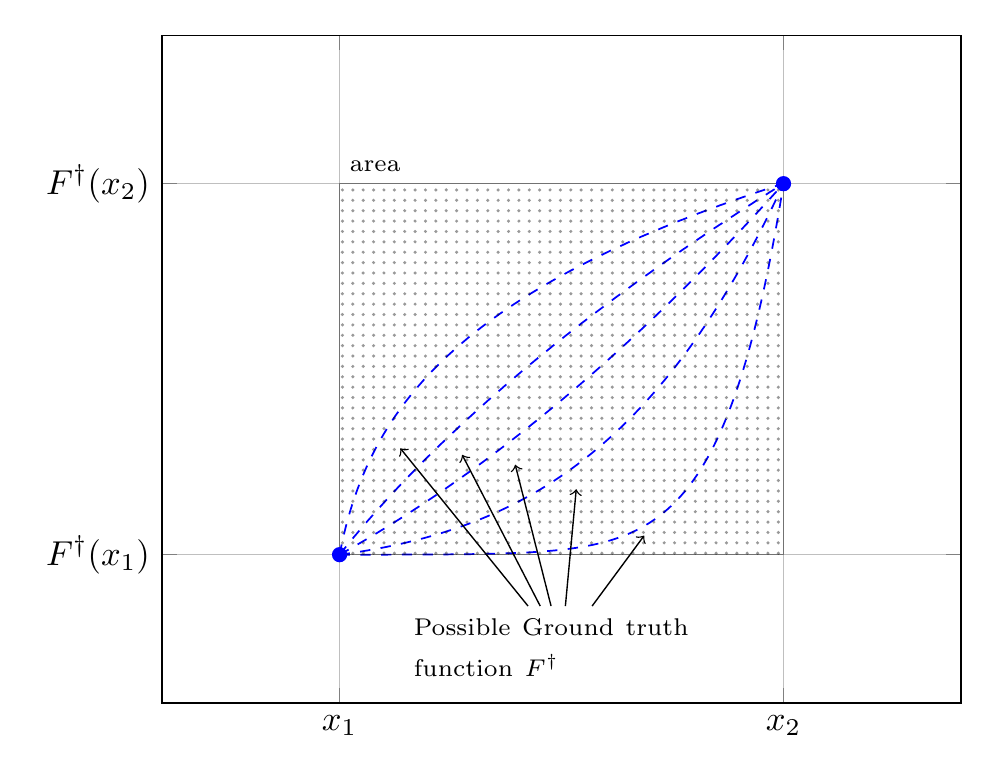}}%
	\subfigure[Right-continuous piecewise constant interpolation function.]{\includegraphics[width=0.45\linewidth]{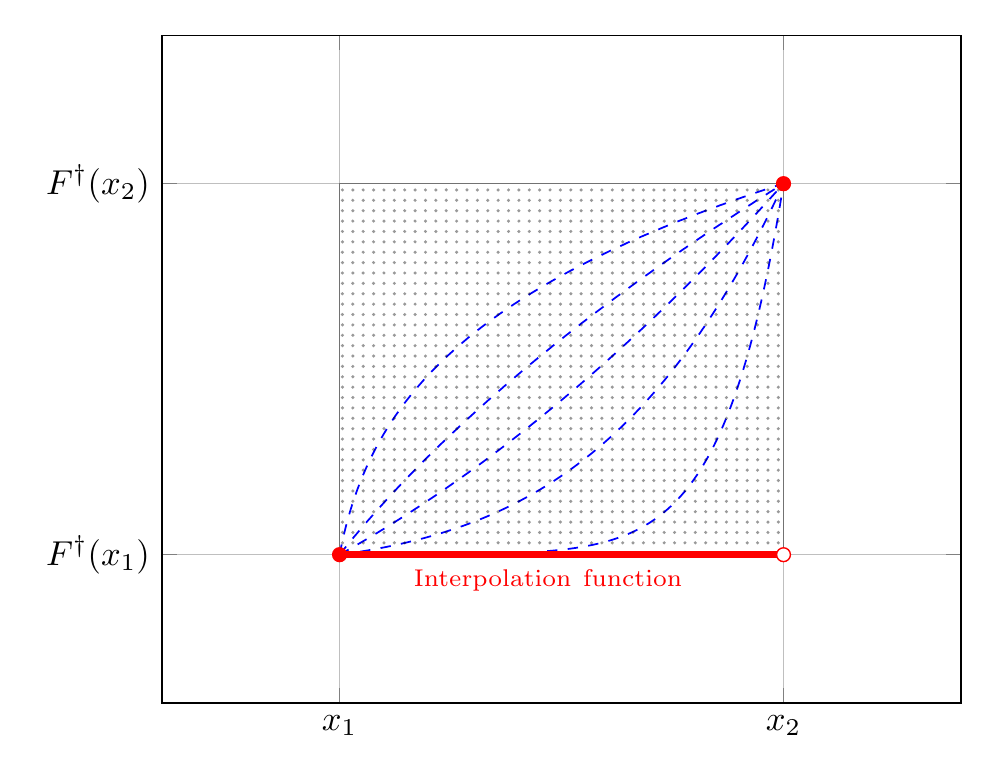}}%
	\caption{Possible ground truth functions between two consecutive points $x_1$ and $x_2$, and our choice of piecewise constant interpolant.}
	\label{fig:possible_func_and_piecewise_constant}
\end{figure}

To adopt a conservative point of view, we choose as the approximating function $\Fapp$ of $\Fdag$ a piecewise constant interpolation function, say:
\begin{equation}
	\Fapp(x) = \sum_{i=1}^{\Nobs-1} \yobs_i \indic_{[x_i, x_{i+1})}(x)\,,
\end{equation}
where $\indic_\mathcal{I}$ denotes the indicator function of the interval $\mathcal{I}$.
We do not want this interpolation function to overestimate the true function $\Fdag$ as one knows that the numerical estimate in our case always underestimates the ground truth function $\Fdag(x)$.
See \Cref{fig:possible_func_and_piecewise_constant} for an illustration of this choice, which can be viewed as a worst-case approach.
Indeed, this chosen interpolation function is the worst possible function underestimating $\Fdag$ given two points $x_1$ and $x_2$ and their respective observations $\yobs_1$ and $\yobs_2$.


\section{Reconstruction algorithms}%
\label{sec:algorithms}%

The reconstruction algorithm that we propose, \Cref{our_algorithm}, is driven to produce a sequences of reconstructions that converges to $\Fdag$ by following a principle of \emph{area minimisation}:
we associate to the discrete data set $\{x_i, \yobs_i\}_{i = 1}^{\Nobs} \subset [a, b] \times \Reals$ a natural notion of area \eqref{eq:area} as explained above, and seek to drive this area towards zero.
\revised{The motivation behind this objective is in \Cref{prop:sqrt_rectangle} which states that the area converges to $0$ as more points are added to the data set.}
\revised{However,} the objective of minimising the area is complicated by the fact that evaluations of $\Fdag$ are imperfect.
Therefore, a key user-defined parameter in the algorithm is $\ER \in (0, \infty)$, which can be thought of as an ``exchange rate'' that quantifies to what extent the algorithm prefers to redo poor-quality evaluations of the target function versus driving the area measure to zero.

\subsection{Algorithm}

The main algorithm is organized as follows, starting from $\step{\Nobs}{0}\geq 2$ points and a dataset that is assumed to be consistent at the initial step $n=0$.
It goes through $\Nstep$ iterations, where $\Nstep$ is either fixed a priori, or obtained a posteriori once a stopping criterion is met.
Note that $\smash{\quality_\text{new}}$ stands for the quality of a newly generated observation $\smash{\yobs_\text{new}}$ for any new point $x_{\text{new}}$ introduced by the algorithm.
The latter is driven by the user-defined ``exchange rate'' $\ER$ as explained just above.
At each step $n$, the algorithm computes the weighted area $\step{\WA}{n}$ as the minimum of the quality times the sum of the areas of the data points:
\begin{equation}
		\step{\WA}{n} =\step{\quality_-}{n}\times\step{\Area}{n}\,,
\end{equation}
where
\begin{equation}\label{eq:qm-area}
		\step{\quality_-}{n} = \min_{1\leq i\leq\step{\Nobs}{n}}\{\step{\quality_i}{n}\}\,,\quad\step{\Area}{n}=\sum_{i=1}^{\step{\Nobs}{n}-1}\step{\area_i}{n}\,,
\end{equation}
$\step{\area_i}{n}$ is the area computed by \Cref{eq:area} at step $n$ (see also \Cref{eq:arean}), and $\smash{\step{\Nobs}{n}}$ is the number of data points.
Then it is divided into two parts according to the value of $\smash{\step{\WA}{n}}$ compared to $\ER$.
\begin{itemize}
	\item If $\smash{\step{\WA}{n} < \ER}$, then the algorithm aims at increasing the quality $\smash{\step{\quality_-}{n}}$ of the worst data point (the one with the lowest quality) with index $\smash{\step{i_-}{n}=\argmin_{1\leq i\leq \step{\Nobs}{n}}\{\step{\quality_i}{n}\}}$ at step $n$.
	It stores the corresponding old value $\smash{\yobs_\text{old}}$, searches for a new value $\smash{\yobs_\text{new}}$ by improving successively the quality of this very point, and stops when $\smash{\yobs_\text{new}} > \smash{\yobs_\text{old}}$.
	\item If $\smash{\step{\WA}{n} \geq \ER}$, then the algorithm aims at driving the total area $\step{\Area}{n}$ to zero.
	In that respect, it identifies the biggest rectangle
	\revised{
	\begin{equation}\label{eq:ap}
	\step{\area_+}{n}= \max_{1\leq i\leq \step{\Nobs}{n}-1}\{\step{\area_i}{n}\}
	\end{equation}
	}
	and its index
	\revised{
	\begin{equation}\label{eq:ip}
	\step{i_+}{n}=\argmax_{1\leq i\leq \step{\Nobs}{n}-1}\{\step{\area_i}{n}\}
	\end{equation}
	}
	and adds a new point $x_{\text{new}}$ at the middle of this biggest rectangle.
	Then, it computes a new data value $\smash{\yobs_\text{new}} = \Fnum(x_{\text{new}},\quality_\text{new})$ with a new quality $\smash{\quality_\text{new}}$.
\end{itemize}
In both cases, \revised{the numerical parameters $\smash{\quality_\text{new}}$ (for example a number of iterations, or the size of a sampling set or a population) are arbitrary and any value can be chosen in practice each time a new point $x_{\text{new}}$ is added to the dataset. They can be increased arbitrarily as well each time such a new point has to be improved}.
\revised{Indeed, the numerical parameters $\quality$ of the optimisation routine we have access to can be increased as much as desired, and increasing them will improve the estimates $\Fnum(x,\quality)$ of the true function $\Fdag(x)$ uniformly in $x$; see \Cref{ass:Fnum2Fdag}}.
The algorithm then verifies the consistency of the dataset by checking the quality of each point.
If there is any inconsistent point, the algorithm computes a new value until obtaining consistency by improving successively the corresponding reliability.
\revised{This is achieved in a finite number of steps starting from an inconsistent point and exploring the dataset from the left to the right.}

Finally, the algorithm updates the quality vector $\smash{\{\step{\quality_i}{n+1}\}_{i=1}^{\step{\Nobs}{n+1}}}$, the area vector $\smash{\{\step{\area_i}{n+1}\}_{i=1}^{\step{\Nobs}{n+1}}}$, the worst quality $\smash{\step{\quality_-}{n+1}}$ and the index $\smash{\step{i_-}{n+1}}$ of the corresponding point, the biggest rectangle $\smash{\step{\area_+}{n+1}}$ and its index $\smash{\step{i_+}{n+1}}$, and then the new weighted area $\smash{\step{\WA}{n+1}}$.

\begin{algorithm}
	\small
	\SetAlgoLined
	\caption{Adaptive algorithm to reconstruct a monotonically increasing function $\Fdag$}\label[algorithm]{our_algorithm}
	\BlankLine
	 \KwIn{$\step{\Nobs}{0} \geq 2$, $\{\step{x_i}{0}, \step{\yobs_i}{0},\step{\quality_i}{0}\}_{i=1}^{\step{\Nobs}{0}}$ and $\ER$.}
	 \KwOut{$\{\step{x_i}{\Nstep}, \step{\yobs_i}{\Nstep},\step{\quality_i}{\Nstep}\}_{i=1}^{\step{\Nobs}{\Nstep}}$ with $\step{\Nobs}{\Nstep} \geq \step{\Nobs}{0}$.}
	  \BlankLine
	\textbf{Initialization}:\\
	 Get the worst quality point and its index\string:
	 \begin{itemize}[leftmargin=*,labelsep=5.8mm]
		\item $\quality_{-}^{(0)} = \underset{1\leq i\leq \step{\Nobs}{0}}{\min}\{\step{\quality_i}{0}\}$;
		\item $i_{-}^{(0)} = \underset{1\leq i\leq \step{\Nobs}{0}}{\argmin}\{\step{\quality_i}{0}\}$.
	 \end{itemize}
	Compute the area of each pair of data points\string:  $\step{\area_i}{0} = (\step{x_{i+1}}{0}- \step{x_i}{0})\times(\step{\yobs_{i+1}}{0} -\step{\yobs_i}{0})$.\\ 
	Get the biggest rectangle and its index\string:  
	 \begin{itemize}[leftmargin=*,labelsep=5.8mm]
		 \item $\area_{+}^{(0)} = \underset{1\leq i\leq \step{\Nobs}{0}-1}{\max}\{\step{\area_i}{0}\}$;
		 \item $i_{+}^{(0)} = \underset{1\leq i\leq \step{\Nobs}{0}-1}{\argmax}\{\step{\area_i}{0}\}$.
	 \end{itemize}
	Define the weighted area at step $n = 0$ as $\step{\WA}{0} = \quality_{-}^{(0)}\times\sum\limits_{i=1}^{\step{\Nobs}{0}-1}\step{\area_i}{0}$.\\
	
	\While{$n\leq\Nstep$}{
	\uIf{$\step{\WA}{n} < \ER$}{
	Data points are unchanged\string: $\step{\Nobs}{n+1} = \step{\Nobs}{n}$ and $\{\step{x_i}{n+1}\}_{i=1}^{\step{\Nobs}{n+1}} = \{\step{x_i}{n}\}_{i=1}^{\step{\Nobs}{n}}$\;
	Store the old value $\yobs_\text{old} = \step{\yobs_{i_{-}^{(n)}}}{n}$\;
	\While{$\yobs_\textup{new} \leq \yobs_\textup{old}$}{Compute a new value $\yobs_\text{new} = \Fnum(\step{x_{i_{-}^{(n)}}}{n},\quality_\text{new})$\;}
	}
	\Else{
	Introduce a new point at the middle of the biggest rectangle\string: $\step{\Nobs}{n+1} = \step{\Nobs}{n} + 1$, $x_\text{new} = \demi(\step{x_{i_{+}^{(n)}}}{n} + \step{x_{i_{+}^{(n)}+1}}{n})$,
	and $(\step{x_1}{n+1},\dots,\step{x_{i_{+}^{(n)}}}{n+1}, \step{x_{i_{+}^{(n)}+1}}{n+1},\step{x_{i_{+}^{(n)}+2}}{n+1}, \dots, \step{x_{\step{\Nobs}{n+1}}}{n+1}) = (\step{x_1}{n},\dots,\step{x_{i_{+}^{(n)}}}{n}, x_\text{new}, \step{x_{i_{+}^{(n)}+1}}{n}, \dots, \step{x_{\step{\Nobs}{n}}}{n})$\;
	Compute the new value $\yobs_\text{new}=\Fnum(x_\text{new},\quality_\text{new})$\;
	}
	Verify consistency of the pointwise observations $\{\step{\yobs_{i}}{n+1})\}_{i=1}^{\step{I}{n+1}}$ by checking their quality. If there are not consistent, recompute them until obtaining consistency and then update the quality vector\;
	Compute the new quality vector $\{\step{\quality_i}{n+1}\}_{i=1}^{\step{\Nobs}{n+1}}$ and area vector $\{\step{\area_i}{n+1}\}_{i=1}^{\step{\Nobs}{n+1}}$\;
	Update $\step{\quality_{-}}{n+1}, \step{i_{-}}{n+1}, \step{\area_+}{n+1}$ and $\step{i_{+}}{n+1}$\;
	Compute $\step{\WA}{n+1} = \step{\quality_{-}}{n+1}\times\sum\limits_{i=1}^{\step{\Nobs}{n+1}-1}\step{\area_i}{n+1}$\;
	$n = n + 1$\;
	}
	\label{algo:b}
\end{algorithm}

\subsection{Proof of convergence}
\label{sec:proof}
We denote by $\step{\Nobs}{n}$ the number of data points, and $\smash{\{\step{x_i}{n},\step{\yobs_i}{n},\step{\quality_i}{n}\}_{i=1}^{\step{\Nobs}{n}}}$ the positions of the data points, the observations given by the optimization algorithm at these positions, and the qualities associated with the optimization algorithm at the step $n$ of \Cref{our_algorithm}.
For each $i =\smash{1,\dots,\step{\Nobs}{n}-1}$, we define $\smash{\step{s_i}{n}} = \smash{[\step{x_i}{n}, \step{x_{i+1}}{n}[} \subset [a,b]$ and the vector containing all rectangle areas $\smash{\{\step{\area_i}{n}\}_{i=1}^{\step{\Nobs}{n}-1}}$ by:
\begin{equation}\label{eq:arean}
\step{\area_i}{n} = (\step{x_{i+1}}{n} - \step{x_i}{n})\times(\step{\yobs_{i+1}}{n} - \step{\yobs_i}{n})\,.
\end{equation}

The pointwise observation $\step{\yobs_i}{n}=\Fnum(\step{x_{i}}{n},\step{\quality_{i}}{n})$ is thus associated to the quality $\step{\quality_i}{n}\in\Reals_+$, which quantifies the confidence we have in this observation as outlined in the problem description in \Cref{sec:setup}.
This number can represent the inverse error achieved by the optimization algorithm, for example, or the number of iterations, or the number of individuals in a population, or any other numerical parameter pertaining to this optimization process.
The higher it is, the closer the observation is to the true target value.
Therefore we consider the following assumption on the numerical process $\Fnum$.
\begin{Assumption}\label{ass:Fnum2Fdag}
	$\Fnum(x,\quality)$ converges to $\Fdag(x)$ as $\quality\to+\infty$ uniformly in $x\in[a,b]$, that is:
	\begin{equation*}\
		\forall\epsilon>0\,,\;\exists Q>0\;\text{such that}\;\forall\quality\geq Q\,,\;\forall x\in[a,b]\,,\;\Absval{\Fnum(x,\quality)-\Fdag(x)}\leq\epsilon\,.
	\end{equation*}
\end{Assumption}

Moreover, we can guarantee that:
\begin{equation}
	\forall x \in[a,b]\,,\quad\forall \quality \in\Reals_+\,,\quad\Fnum(x,\quality) \leq \Fdag(x).
\end{equation} 

That is, the optimisation algorithm will always underestimate the true value $\Fdag(x)$.
In this way, one can model the relationship between the numerical estimate $\Fnum$ and the true value $\Fdag$ as:
\begin{equation}
	\label{eq:underestimate}
	\forall x \in [a,b]\,, \quad\forall \quality \in\Reals_+\,,\quad \Fnum(x,\quality) = \Fdag(x) - \epsilon(x,\quality)\,,
\end{equation}
where $\epsilon$ is a positive random variable.
\revised{These assumptions imply some robustness and stability of the algorithm we use}.

In the following, we will assume that $\step{I}{0}\geq2$.
That is, we have at least two data points at the beginning of the reconstruction algorithm.
Also among these points, we have one point at $x=a$ and another one at $x=b$.
Moreover, we will assume that the initial dataset is consistent.
Since \Cref{our_algorithm} recomputes the inconsistent points at all steps, we can also consider in the following that any new numerical observation is actually consistent.
\revised{Also, we need to guarantee that the weighted area $\step{\WA}{n}$ will permanently oscillate about $\ER$ as the iteration step $n$ is increasing; this is the purpose of \Cref{ass:q_increasing} below as shown in the subsequent \Cref{prop:Ep}. From these properties it will then be shown that \Cref{our_algorithm} is convergent, as stated in \Cref{thm:pw_cvg}}.

\begin{Assumption}\label{assump:consistent}
	Any new numerical value obtained by \Cref{our_algorithm} is consistent.  
\end{Assumption}
\begin{Assumption}\label{ass:q_increasing}
\revised{$\step{\quality_-}{n}\to+\infty$ as $n \to \infty$}.
\end{Assumption}

Within \Cref{assump:consistent} all points have a consistency of $1$, and therefore $\quality=r>0$ the reliability.
Besides, one has $\smash{\Fnum(\step{x_i}{n},\step{\quality_i}{n})\leq\Fnum(\step{x_{i+1}}{n},\step{\quality_{i+1}}{n})}$, that is, $\smash{\step{\yobs_i}{n}\leq\step{\yobs_{i+1}}{n}}$ for all points $i$ and steps $n$.
We finally define the sequence of piecewise constant reconstruction functions $\step{\Fapp}{n}$ as follows.
\begin{Definition} \label{def:piewise_constant_funct}
	For each $x \in [a,b]$, we define the reconstructing function $\step{\Fapp}{n}$ at step $n$ as:
	\begin{equation*}
	\step{\Fapp}{n}(x) = \sum_{i=1}^{\step{\Nobs}{n}-1}\step{\yobs_i}{n}\indic_{\step{s_i}{n}}(x)\,, 
	\end{equation*}
	and $\smash{\step{\Fapp}{n}(\step{x_{\step{\Nobs}{n}}}{n})} = \step{\Fapp}{n}(b)=\step{\yobs_{\step{\Nobs}{n}}}{n}$.
\end{Definition}

\revised{Now let
\begin{align}
	E^+ & \defeq \{n\in\Naturals\,;\;\step{\WA}{n}\geq\ER\}\,, &
	E^- & \defeq \{n\in\Naturals\,;\;\step{\WA}{n}<\ER\}\,,
\end{align}
which are such that $E^+\cup E^-=\Naturals$ and $E^+\cap E^-=\varnothing$. In order to prove the convergence (in a sense to be given) of \Cref{our_algorithm}, we first need to establish the following intermediate results, \Cref{prop:Ep}, \Cref{prop:sqrt_rectangle}, and \Cref{prop:Em}. They clarify the behaviour of the sequence $\smash{\step{\WA}{n}}$ when points are added to the dataset and the largest area $\smash{\step{\area_+}{n}}$ is divided into four parts at each iteration step $n$; see \Cref{fig:area_proof}}.

\begin{Proposition}\label{prop:Ep}
\revised{$E^+$ is infinite.}
\end{Proposition}

\begin{proof}
\revised{Let us assume that $E^+$ is finite: $\exists N$ such that $\forall n\geq N$, $n\in E^-$. Therefore we are in the situation $\smash{\step{\WA}{n}}<\ER$, the minimum quality $\step{\quality_-}{n}$ of the data goes to infinity, and the total area $\smash{\step{\Area}{n}}$ is modified although the evaluation points $\smash{\{\step{x_i}{n}\}_{i=1}^{\step{I}{n}}}$ and their number $\smash{\step{I}{n}}$ are unchanged; thus they are independent of $n$.
Repeating this step yields
\begin{equation*}	
	\lim_{n \to \infty}\step{\Area}{n}=\sum_{i=1}^{I-1}(x_{i+1}-x_i)(\Fdag(x_{i+1})-\Fdag(x_i))=\Area>0
\end{equation*}
since $\Fdag$ is monotonically increasing and non-constant on $[a,b]$, and \Cref{ass:Fnum2Fdag} is used. Consequently $\smash{\step{\WA}{n}\to+\infty}$ as $n \to \infty$, that is $\smash{\step{\WA}{n}}\geq\ER$ $\forall n\geq N_1$ for some $N_1$, which is a contradiction.}
\end{proof}

\revised{The set $E^+$ is therefore of the form
\begin{equation*}
	E^+=\bigcup_{k\geq1}\bbracket{m_k,n_k}\,,
\end{equation*}
where
\begin{equation*}
	\bbracket{m_k,n_k} \defeq \{n\in\Naturals\,;\;m_k\leq n\leq n_k\}\,.
\end{equation*}

Let us introduce the strictly increasing application $\subseq:\Naturals\to\Naturals$ such that $\subseq(p)$ is the $p$\textsuperscript{th} element of $E^+$ (in increasing order), and $\bbracket{m_k,n_k}=\subseq(\bbracket{p_k+1,p_{k+1}})$. $p$ is the counter of the elements of $E^+$, and $n$ is the corresponding iteration number.}

\begin{Proposition}\label{prop:sqrt_rectangle}
\revised{Let $\step{\Nobs}{\subseq(p)}=\step{\Nobs}{\subseq(0)}+p$.
Then
	\begin{equation*}
		\step{\Area}{\subseq(p)} = \sum_{i=1}^{\step{\Nobs}{\subseq(p)}-1} \step{\area_i}{\subseq(p)} = \mathrm{O}\left(\frac{1}{\sqrt{p}}\right)
	\end{equation*}
as $p \to \infty$, and $\step{\Area}{n}\to 0$ as $n\to 0$.}
\end{Proposition}

\begin{proof}
\revised{Let $k\geq 1$ and $n=\subseq(p)\in\smash{\bbracket{m_k,n_k}}$, where $p\in\smash{\bbracket{p_k+1,p_{k+1}}}$. Let $\smash{\step{\Area}{n}}$ be given by \Cref{eq:qm-area}, $\smash{\step{\area_+}{n}}$ be given y \Cref{eq:ap}, and $\smash{\step{i_+}{n}}$ be given by \Cref{eq:ip}}.
	At iteration $n+1$ one has:
	\begin{equation*}
		\step{x_i}{n+1}=
		\begin{cases}
			\step{x_i}{n} & \text{for $1\leq i \leq\step{i_+}{n}$,} \\
			\dfrac{1}{2} \left(\step{x_{\step{i_+}{n}}}{n}+\step{x_{\step{i_+}{n}+1}}{n}\right) & \text{for $i=\step{i_+}{n}+1$,} \\
			\step{x_{i-1}}{n} & \text{for $\step{i_+}{n}+2\leq i \leq\step{I}{n+1}$.}
		\end{cases}
	\end{equation*}
	
	\begin{figure}
		\centering
		\includegraphics[width=0.50\linewidth]{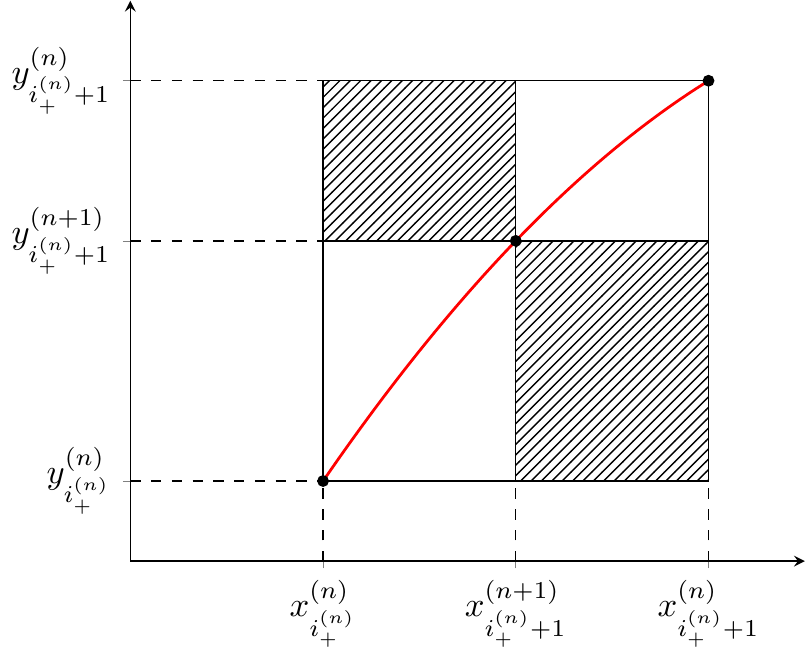}
		\caption{New area when one adds a point at the middle of the biggest rectangle.}
		\label{fig:area_proof}
	\end{figure}
	
	Also $\step{\yobs_i}{n+1}\leq\step{\yobs_{i+1}}{n+1}$ for $1\leq i\leq\step{I}{n+1}-1$.
	One may check that $\step{\area_+}{n}=2\step{\area_{\step{i_+}{n}}}{n+1}+2\step{\area_{\step{i_+}{n}+1}}{n+1}$ (see \Cref{fig:area_proof}) and therefore:
	\begin{equation}
		\label{eq:Anp1}
		\step{\Area}{n+1} =\step{\Area}{n}-\step{\area_+}{n}+\step{\area_{\step{i_+}{n}}}{n+1}+\step{\area_{\step{i_+}{n}+1}}{n+1} \\
		=\step{\Area}{n}-\demi\step{\area_+}{n}\,.
	\end{equation}
	Besides $\step{\Area}{n}\leq(\step{\Nobs}{n}-1)\step{\area_+}{n}$ so that one has:
	\begin{align*}
		\step{\Area}{n+1}
		&\leq \step{\Area}{n} -\frac{\step{\Area}{n}}{2(\step{\Nobs}{n}-1)} \\
		&\leq \step{\Area}{n} \left(\frac{2(\step{\Nobs}{n}-1)-1}{2(\step{\Nobs}{n}-1)}\right)\,,
	\end{align*}
\revised{or:}
\begin{equation}\label{eq:interm1}
\revised{\step{\Area}{\subseq(p)+1}\leq\step{\Area}{\subseq(p)} \left(\frac{2(\step{\Nobs}{\subseq(p)}-1)-1}{2(\step{\Nobs}{\subseq(p)}-1)}\right)\,.}
\end{equation}
\revised{At this stage two situations arise:
\begin{itemize}
\item either $p\in\smash{\bbracket{p_k+1,p_{k+1}-1}}$, in which case $\subseq(p)+1=\subseq(p+1)$;
\item or $p=p_{k+1}$, in which case by our algorithm $\step{\Area}{n}$ is kept constant from $n=n_k+1$ to $n=m_{k+1}$; that is $\smash{\step{\Area}{{n_k}+1}}=\smash{\step{\Area}{{m_{k+1}}}}$, or:
\begin{equation*}
\revised{\step{\Area}{\subseq({p_{k+1}})+1}=\step{\Area}{\subseq({p_{k+1}}+1)}\,.}
\end{equation*}
\end{itemize}}
\revised{The choice of $k$ being arbitrary, one concludes that \Cref{eq:interm1} also reads $\forall p\in\Naturals$:
\begin{align*}
\step{\Area}{\subseq(p+1)} &\leq\step{\Area}{\subseq(p)} \left(\frac{2(\step{\Nobs}{\subseq(p)}-1)-1}{2(\step{\Nobs}{\subseq(p)}-1)}\right) \\
&\leq\step{\Area}{\subseq(p)} \left(\frac{2(\step{\Nobs}{\subseq(0)}+p-1)-1}{2(\step{\Nobs}{\subseq(0)}+p-1)}\right)\,.
\end{align*}}
\revised{Thus:
	\begin{align*}
		\step{\Area}{\subseq(p)}
		&\leq \step{\Area}{\subseq(1)} \prod_{i=1}^{p-1}\left(\frac{2(\step{\Nobs}{\subseq(0)}+i-1)-1}{2(\step{\Nobs}{\subseq(0)}+i-1)}\right) \\
		&\leq \step{\Area}{\subseq(1)} \prod_{i=1}^{p-1}\left(\frac{1+\frac{\alpha}{i}}{1+\frac{\beta}{i}}\right)\,,
	\end{align*}	
	letting $\alpha=\step{\Nobs}{\subseq(0)}-\frac{3}{2}$ and $\beta=\step{\Nobs}{\subseq(0)}-1$.}
	However,
	\begin{equation*}
		\sum_{i=1}^p\log\left(1+\frac{\alpha}{i}\right)=\alpha\sum_{i=1}^p\frac{1}{i}+\revised{C''_p}
	\end{equation*}
	where \revised{$\lim_{p \to \infty}C''_p=C''$}, and
	\begin{equation*}
		\sum_{i=1}^p\frac{1}{i} = \log p+\gamma+\revised{\epsilon'_p}\,,
	\end{equation*}
	where $\gamma$ is the Euler constant and \revised{$\lim_{p \to \infty} \epsilon'_p=0$}.
	Consequently:
	\begin{align*}
		\sum_{i=1}^{p-1}\log\left(1+\frac{\alpha}{i}\right)-\sum_{i=1}^{p-1}\log\left(1+\frac{\beta}{i}\right) &= (\alpha-\beta)\log(p-1) + C'_p \\
		&=(\alpha-\beta)\left[\log p+\log\left(1-\frac{1}{p}\right)\right]+C'_p  \\
		&=\log\left(\frac{1}{\sqrt{p}}\right)+C_p\,,
	\end{align*}
	since $\alpha-\beta=-\demi$;
	again \revised{$C_p$ and $C'_p$ are sequences with constant limits $\lim_{p \to \infty}C_p=C$ and $\lim_{p \to \infty}C'_p=C'$}.
	Therefore,
	\begin{equation*}
		\prod_{i=1}^{p-1} \left( \frac{1+\frac{\alpha}{i}}{1+\frac{\beta}{i}} \right) = \frac{\revised{\mathcal{C}}}{\sqrt{p}}(1+\epsilon_p)
	\end{equation*}
	where \revised{$\mathcal {C}$} is a constant, and $\lim_{p \to \infty}\epsilon_p=0$.
	\revised{One also concludes that $\step{\Area}{n}$, which is either kept constant or equal to $\step{\Area}{\subseq(p)}$, converges to $0$ as $n\to\infty$}.
	Hence the claimed results hold.
\end{proof}

\begin{Proposition}\label{prop:Em}
\revised{$E^-$ is infinite.}
\end{Proposition}

\begin{proof}
\revised{Let us assume that $E^-$ is finite: $\exists N$ such that $\forall n\geq N$, $n\in E^+$. Therefore we are in the situation $\smash{\step{\WA}{n}}\geq\ER>0$, and $\subseq(n)$ has the form $\subseq(n)=\smash{n-n_0}$, $n\geq N$ for some $\smash{n_0}\in\Naturals$. From \Cref{prop:sqrt_rectangle}:
\begin{equation*}
\step{\Area}{n-n_0}=\mathrm{O}\left(\frac{1}{\sqrt{n}}\right)\,,
\end{equation*}
thus $\smash{\step{\Area}{n}}\to 0$ and $\smash{\step{\WA}{n}}\to 0$ as $n\to\infty$ since $\smash{\step{\quality_-}{n}}$ is kept unchanged, which is a contradiction.}
\end{proof}

We now provide three results on the convergence of \Cref{our_algorithm}.
As is to be expected, the algorithm can only be shown to converge uniformly when the target response function $\Fdag$ is sufficiently smooth; otherwise, the convergence is at best pointwise or in mean.

\begin{Theorem}[Algorithm convergence]\label{thm:pw_cvg}
	Assume that $\Fdag$ is strictly increasing.
	Then, for any choice of $\ER > 0$, \Cref{our_algorithm} is convergent in the following senses:
	\begin{itemize}
		\item If $\Fdag$ is piecewise continuous on $[a,b]$, then $\smash{\lim_{n \to \infty}\step{\Fapp}{n}(x)} = \smash{\Fdag(x)}$ at all points $x\in[a,b]$ where $\Fdag$ is continuous;
		\item If $\Fdag$ is continuous on $[a,b]$, then convergence holds uniformly: $\|\step{\Fapp}{n} - \Fdag\|_\infty \xrightarrow[n \to \infty]{} 0$.
	\end{itemize}
\end{Theorem}

\begin{proof}
	\revised{Let $\ER > 0$. We know from \Cref{prop:Ep,prop:Em} that $\step{\WA}{n}$ will oscillate about $\ER$ in the iterating process as $n\to\infty$, while $\smash{\lim_{n \to \infty}\step{\quality_-}{n}} = +\infty$ from \Cref{ass:q_increasing}. Furthermore, let}
	\begin{equation*}
	\revised{\step{\Delta}{n} \defeq \sup_{1\leq i\leq\step{\Nobs}{n}-1}\Absval{\step{x_{i+1}}{n}-\step{x_i}{n}}\,.}
	\end{equation*}
	\revised{Assuming for example that for some $j$, $\smash{\step{s_j}{n}}=\smash{[\step{x_j}{n},\step{x_{j+1}}{n})}$ is never divided in two in the iteration process and is thus independent of $n$, it turns out that $\smash{\step{\area_j}{n}} \to \smash{(x_{j+1}-x_{j})(\Fdag(x_{j+1})-\Fdag(x_j))}>0$ as $n \to \infty$, which is impossible because $\smash{\step{\Area}{n}}$ goes to $0$ as $n \to \infty$ from \Cref{prop:sqrt_rectangle}. Therefore there exists some $m\in\Naturals^*$ (depending on $n$) such that $\smash{\step{\Delta}{n+m}}\leq\smash{\demi\step{\Delta}{n}}$; also the sequence $\smash{\step{\Delta}{n}}$ is decreasing, hence $\smash{\step{\Delta}{n}} \to 0$ as $n \to \infty$.}
\revised{Now let $x\in[\step{x_i}{n},\step{x_{i+1}}{n})$. Then:
	\begin{align*}
		\Absval{\step{\Fapp}{n}(x)-\Fdag(x)} &= \Absval{\Fnum(\step{x_i}{n},\step{\quality_i}{n})-\Fdag(x)} \\
		&\leq \Absval{\Fnum(\step{x_i}{n},\step{\quality_i}{n})-\Fdag(\step{x_i}{n})}+\Absval{\Fdag(\step{x_i}{n})-\Fdag(x)}\,.
	\end{align*}	
	But $\smash{\step{x_i}{n}}\to x$ as $n \to \infty$ because $\smash{\step{\Delta}{n}}\to 0$; thus if $\Fdag$ is continuous at $x$, the second term on the right hand side above goes to $0$ as $n \to \infty$. However, if $\Fdag$ is continuous everywhere on $[a,b]$, it is in addition uniformly continuous on $[a,b]$ by Heine's theorem, and the second term goes to $0$ as $n \to \infty$ uniformly on $[a,b]$.
	Finally, invoking \Cref{ass:Fnum2Fdag}, the first term on the right hand side above also tends to $0$ as $n \to \infty$.
	This completes the proof.}
\end{proof}

\begin{Proposition}[Convergence in mean]\label{prop:convmean}
	Let $\Fdag\colon[a,b]\to\Reals$ be piecewise continuous.
	Then \Cref{our_algorithm} is convergent in mean in the sense that
	\begin{equation*}
		\|\step{\Fapp}{n} - \Fdag\|_{1}  \xrightarrow[n \to \infty]{} 0.
	\end{equation*}
\end{Proposition}

\begin{proof}
	We can check that the sequence $\step{\Fapp}{n}$ is monotone.
	Indeed, if $\step{\WA}{n} < \ER$, then by construction we have
	\begin{equation*}
	\step{\Fapp}{n+1}(x) - \step{\Fapp}{n}(x)\geq\left(\step{\yobs_{\step{i_-}{n}}}{n+1} - \step{\yobs_{\step{i_-}{n}}}{n}\right)\indic_{\step{s_-}{n}}(x) \geq 0
	\end{equation*}
	where $\step{s_-}{n} = \bigl[ \step{x_{\step{i_-}{n}}}{n}, \step{x_{\step{i_-}{n}+1}}{n} \bigr)$.
	However, if $\step{\WA}{n} > \ER$, then consistency implies that
	\begin{equation*}
		\step{\Fapp}{n+1}(x) - \step{\Fapp}{n}(x)\geq\left(\step{\yobs_{\step{i_+}{n}+1}}{n+1}-\step{\yobs_{\step{i_+}{n}}}{n}\right)\indic_{\step{s_+}{n+1}}(x)\geq0
	\end{equation*}
	where $\step{s_+}{n+1} = \bigl[ \step{x_{\step{i_+}{n}+1}}{n+1}, \step{x_{\step{i_+}{n}+2}}{n+1} \bigr)$.
	The claim now follows from the monotone convergence theorem and the fact that $\step{\Fapp}{0}$ is integrable. 
\end{proof}

\section{Test cases}%
\label{sec:tests}%

To show the effectiveness of \Cref{our_algorithm}, we try it on two cases, in which $\Fdag$ is a continuous function and a discontinuous function respectively.
For both cases, the error between the numerical estimate and the ground truth function is modelled as a random variable following a Log-normal distribution.
That is,
\begin{equation}
	\forall x \in [a,b],\;\epsilon(x) \sim \mathrm{Log}\mathcal{N}(\mu(x),\sigma^{2}),
\end{equation}
with $\sigma^{2} = 1$ and $\mu(x)$ is chosen as $\mathbb{P}[0 \leq \epsilon(x) \leq 0.1\cdot \Fdag(x)] = 0.9$.
Thus, the mean $\mu$ is different for each $x \in [a,b]$.

As we have access to the ground truth function and for validation purpose, the quality value associated to a numerical point is the inverse of the relative error.
Moreover, we assume that the initial points are consistent.

For illustrative purposes, we set the parameter $\ER = 15$ for the examples considered below.

\subsection{\texorpdfstring{$\Fdag$}{F-dagger} is a continuous function}
\label{Fcon}

First, consider the function $\Fdag \in C^{0}([1,2], [1,2])$ defined as follows:
\begin{equation*}
	\Fdag(x)=\begin{cases}
	\Fdag_1(x) & \mbox{if }x \in[1,\frac{3}{2}]\,, \\ \Fdag_2(x) & \mbox{if }x \in[\frac{3}{2},2]\,,
	\end{cases}
\end{equation*}
with
\begin{align}
	\label{eq:F1F2}
	F^{\dagger}_{1}(x) &= a_1\exp(x^{3}) + b_1\,, \\
	\notag
	F^{\dagger}_{2}(x) &= a_2\exp((3-x)^{3}) + b_2\,,
\end{align}
where:
\begin{equation*}
	a_1= -\frac{1}{2(\exp(1) - \exp(27/8))}\,,\quad b_1 = \frac{3 - 2\exp(19/8)}{2(1-\exp(19/8))}\,,\quad a_2= -a_1\,,\quad b_2 = 2a_1\exp(27/8) + b_1\,.
\end{equation*}


The target function $\Fdag$ and the reconstructions $\step{\Fapp}{n}$ obtained through the algorithm for several values of the step $n$ are shown on \Cref{fig:Cont_RecFunc}.
For each $n$, the reconstruction $\step{\Fapp}{n}$ is increasing and the initial points are consistent.
The $\infty$-norm and $1$-norm of the error appear to converge to zero with approximate rates $-0.512$ and $-0.534$ respectively.

\begin{figure}
	\centering
	\subfigure[$n = 0$]{\includegraphics[width=0.30\textwidth]{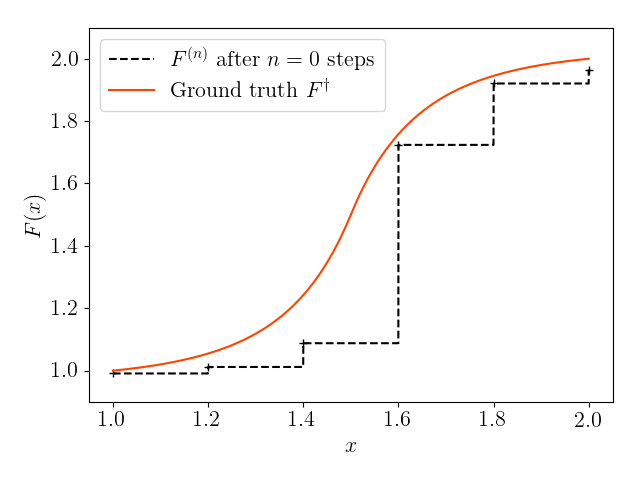}\label{fig:RecFunc_con_step0} }
	\subfigure[$n = 10$]{\includegraphics[width=0.30\textwidth]{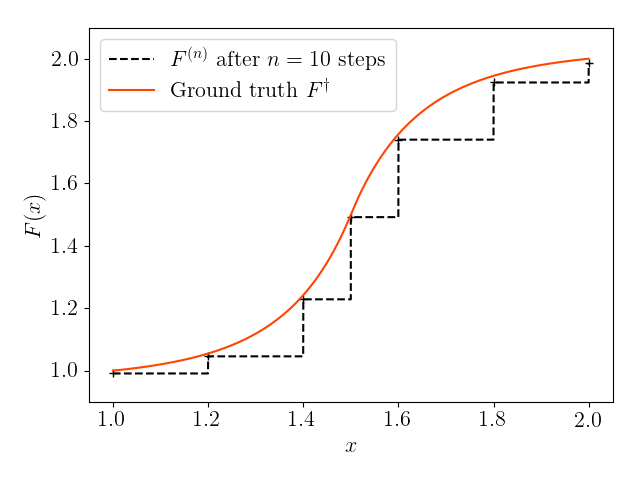}} 
	\subfigure[$n = 100$]{\includegraphics[width=0.30\textwidth]{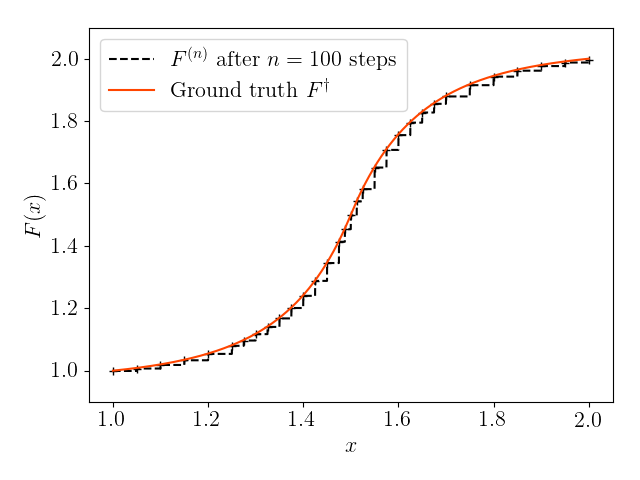}}
	\subfigure[$n = 300$]{\includegraphics[width=0.30\textwidth]{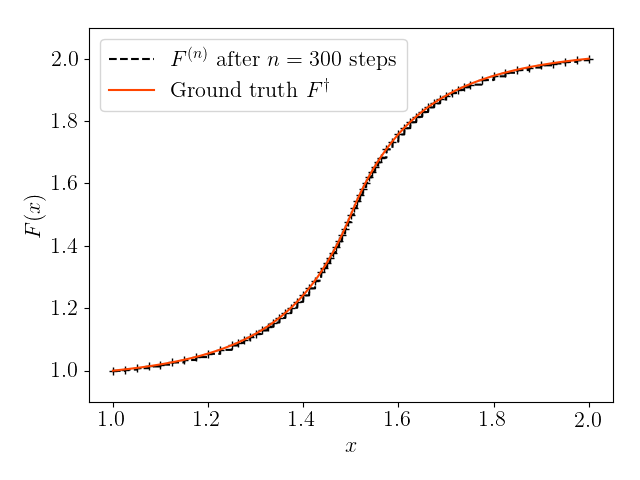}}
	\subfigure[$\infty$-norm error]{\includegraphics[width=0.30\textwidth]{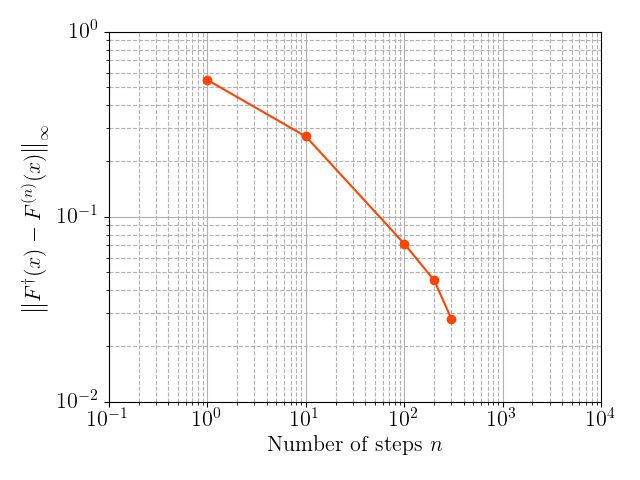}} 
	\subfigure[$1$-norm error]{\includegraphics[width=0.30\textwidth]{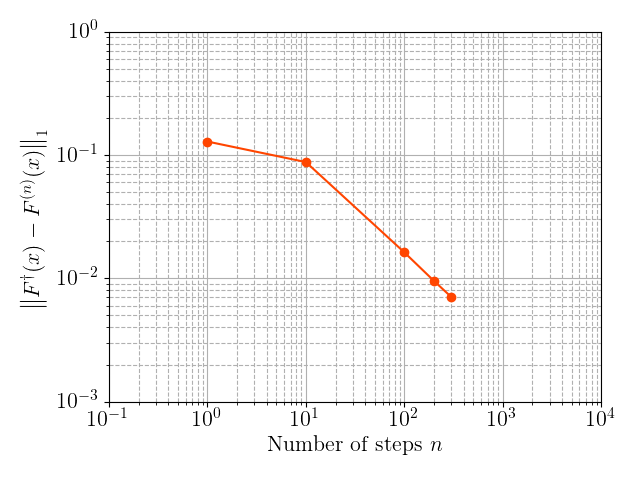}}
	\caption{Evolution of $\step{\Fapp}{n}$ and the $\infty$- and $1$-norms of the error $\Fdag - \step{\Fapp}{n}$ as functions of the iteration count, $n$, for a smooth ground truth $\Fdag$.}
	\label{fig:Cont_RecFunc}
\end{figure}


\subsection{\texorpdfstring{$\Fdag$}{F-dagger} is a discontinuous function}
\label{Fdiscon}

Now, consider the discontinuous function $\Fdag$ defined as follows:
\begin{equation*}
	\Fdag(x)=\begin{cases}
	\Fdag_1 & \mbox{if }x \in[1,\frac{3}{2}]\,, \\ \Fdag_2 & \mbox{if }x \in (\frac{3}{2},2]\,,
	\end{cases}
\end{equation*}
where $\smash{\Fdag_1}$ and $\smash{\Fdag_2}$ are given by \eqref{eq:F1F2}, and:
\begin{align*}
	a_1 & = -\frac{1}{2(\exp(1) - \exp(27/8))}\,, &
	b_1 & = \frac{3 - 2\exp(19/8)}{2(1-\exp(19/8))}\,, \\
	a_2 & = \frac{2}{5(\exp(8) - \exp(27/8))}\,, &
	b_2 & = \frac{ 10 - 8\exp(37/8)}{5(1-\exp(37/8))}\,.
\end{align*}

\begin{figure}[t]
	\centering
	\subfigure[$n = 0$]{\includegraphics[width=0.30\textwidth]{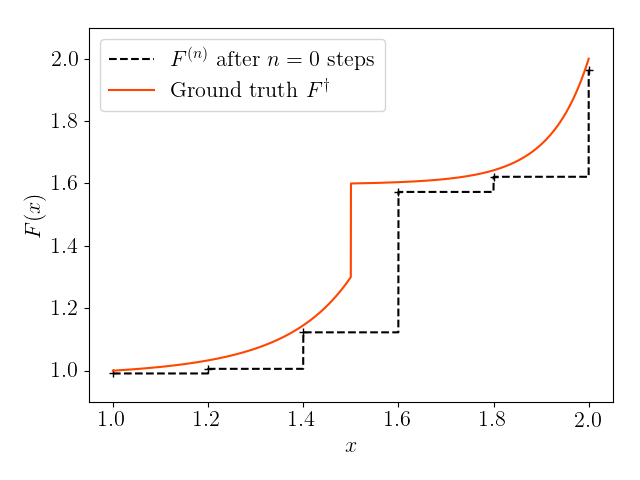}} 
	\subfigure[$n = 10$]{\includegraphics[width=0.30\textwidth]{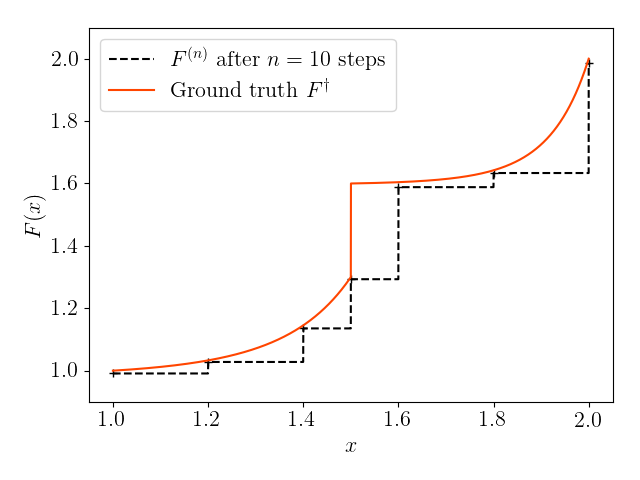}} 
	\subfigure[$n = 100$]{\includegraphics[width=0.30\textwidth]{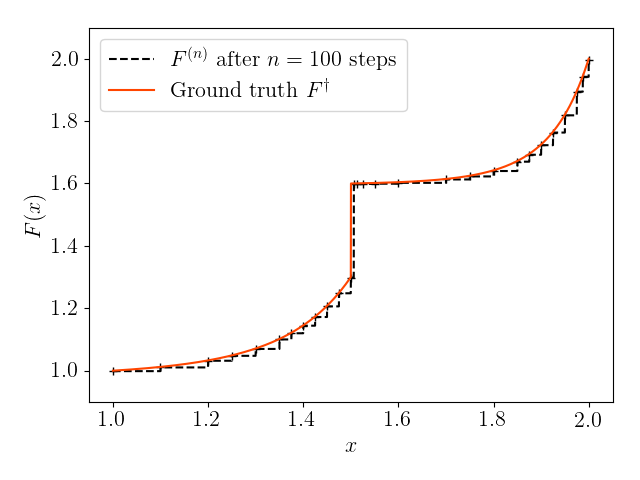}}
	\subfigure[$n = 300$]{\includegraphics[width=0.30\textwidth]{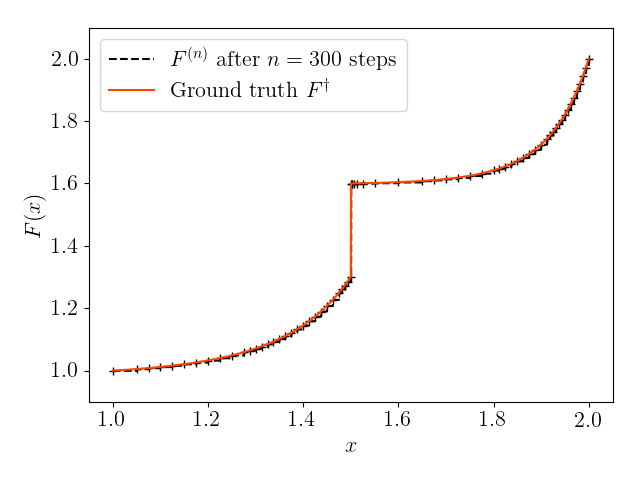}}
	\subfigure[$\infty$-norm error]{\includegraphics[width=0.32\textwidth]{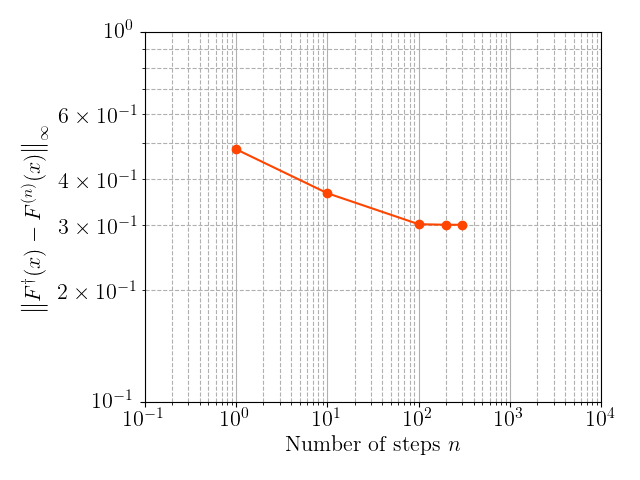}} 
	\subfigure[$1$-norm error]{\includegraphics[width=0.32\textwidth]{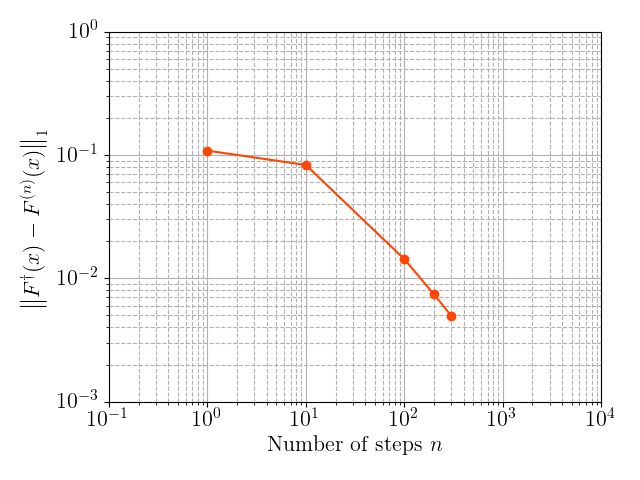}}
	\caption{Evolution of $\step{\Fapp}{n}$ and the $\infty$- and $1$-norms of the error $\Fdag - \step{\Fapp}{n}$ as functions of the iteration count, $n$, for a discontinuous ground truth $\Fdag$.}
	\label{fig:Disc_RecFunc}
\end{figure}

Here, $\Fdag$ is piecewise continuous on $[1,\frac{3}{2}]$ and $]\frac{3}{2},2]$.
In this case, one can apply \Cref{prop:convmean}.
The target function $\Fdag$ and the reconstructions $\step{\Fapp}{n}$ obtained through the algorithm for several values of the step $n$ are shown on \Cref{fig:Disc_RecFunc}.
Observe that the approximation quality, as measured by the $\infty$-norm of the error $\Fdag - \step{\Fapp}{n}$, quite rapidly saturates and does not converge to zero.
This is to be expected for this discontinuous target $\Fdag$, since closeness of two functions in the supremum norm mandates that they have approximately the same discontinuities in exactly the same places.
The $1$-norm error, in contrast, appears to converge at the rate $-0.561$.

\revised{Regarding computational cost, the number of calls to the numerical model is lower when $\Fdag$ is continuous than when it is discontinuous. For both examples above and for the same number of data points, the number of evaluations of the numerical model (analytical formula in the present case) in the discontinuous case is about six times higher than the number of evaluations in the continuous case. This is because the algorithm typically adds more points near discontinuities and the effort of making them consistent increases the number of calls to the model}.

\subsection{Influence of the user-defined parameter \texorpdfstring{$\ER$}{E}}

We consider the case in which $\Fdag$ is discontinuous, as in \Cref{Fdiscon}.
We will show the influence of the choice of the parameter $\ER$ on the reconstruction function $\step{\Fapp}{n}$.

\subsubsection{Case \texorpdfstring{$\ER \ll 1$}{E much less than 1}}
Let us consider the case $\ER = 10^{-4} \ll 1$.
This choice corresponds to the case where one wishes to split over redo the worst quality point.
This can be seen on \Cref{fig:Cont_ER_LOWER} where the worst quality is almost constant over 100 steps while the sum of areas strongly decreases; see \Cref{fig:Cont_ER_LOWER}\subref{fig:minqual_LOWER} and \Cref{fig:Cont_ER_LOWER}\subref{fig:sumarea_LOWER} respectively.
At each step, the algorithm is adding a new point by splitting the biggest rectangle.
One can note on \Cref{fig:Cont_ER_LOWER}\subref{fig:sumarea_LOWER} that the minimum of the quality is not constant.
It means that when the algorithm added a new data point, the point with the worst quality was not consistent any more and had to be recomputed.
In summary, in this case, we obtain more points but with lower quality values.

\begin{figure}[t]
	\centering
	\subfigure[$n = 0$]{\includegraphics[width=0.32\textwidth]{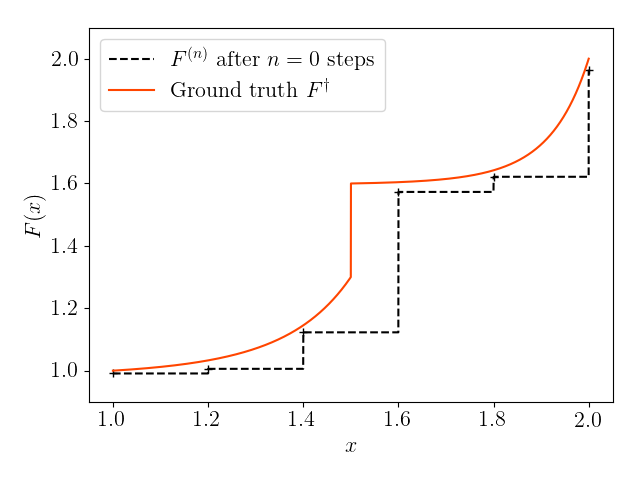}}
	\subfigure[$n = 10$]{\includegraphics[width=0.32\textwidth]{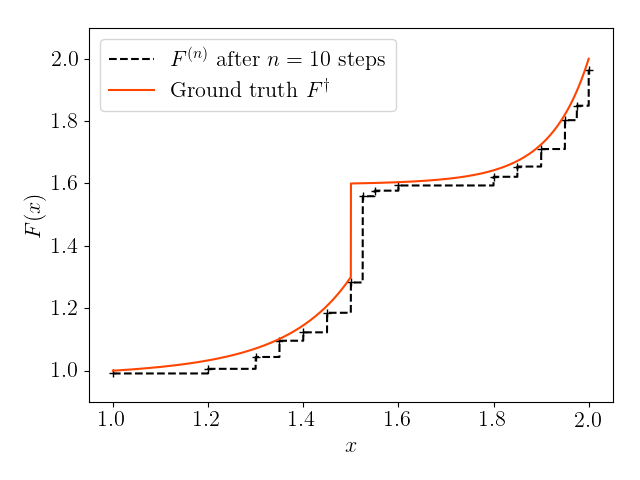}}
	\subfigure[$n = 50$]{\includegraphics[width=0.32\textwidth]{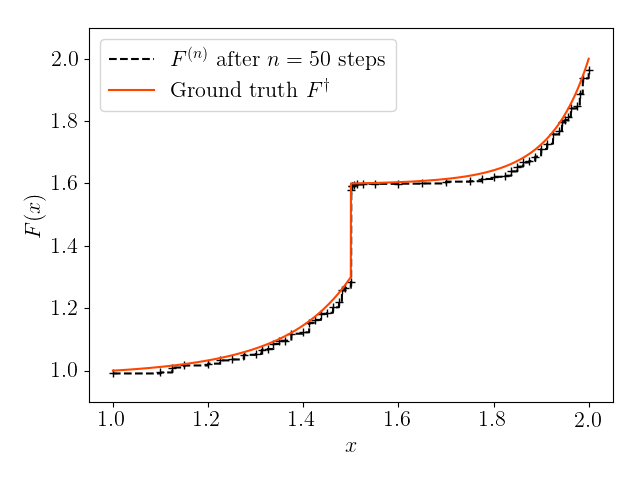}}
	\subfigure[$n = 100$]{\includegraphics[width=0.32\textwidth]{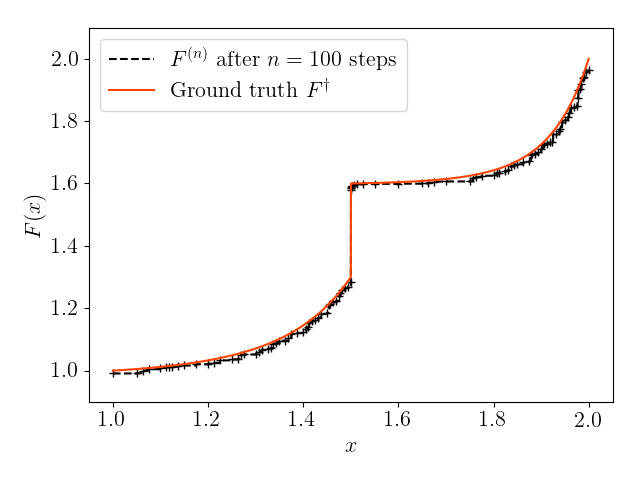}}
	\subfigure[Minimum of the quality]{\includegraphics[width=0.32\textwidth]{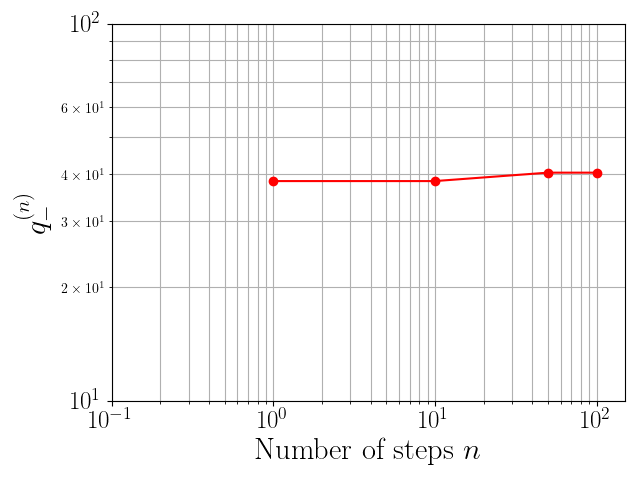}\label{fig:minqual_LOWER} }
	\subfigure[Total area]{\includegraphics[width=0.32\textwidth]{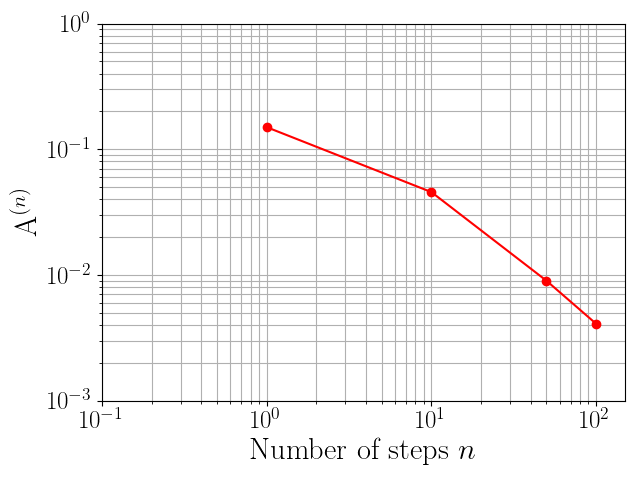}\label{fig:sumarea_LOWER} }
	\caption{Evolution of $\step{\Fapp}{n}$ and the minimum of the quality and the total area as functions of the iteration count, $n$, for a discontinuous ground truth $\Fdag$ with $\ER = 10^{-4}$.}
	\label{fig:Cont_ER_LOWER}
\end{figure}

\subsubsection{Case \texorpdfstring{$\ER \gg 1$}{E much greater than 1}}

We now consider the case $\ER = 10^{4} \gg 1$.
This choice corresponds to the case where one wishes to redo the worst quality point over split.
This can be seen on \Cref{fig:Cont_ER_HIGHER} where the sum of areas stays more or less the same over $100$ steps while the minimum of the quality surges; see \Cref{fig:Cont_ER_HIGHER}\subref{fig:sumarea_HIGHER} and \Cref{fig:Cont_ER_HIGHER}\subref{fig:minqual_HIGHER} respectively.
There is no new point.
The algorithm is only redoing the worst quality point to improve it.
To sum up, we obtain fewer points with higher quality values.

\begin{figure}[t]
	\centering
	\subfigure[$n = 0$]{\includegraphics[width=0.32\textwidth]{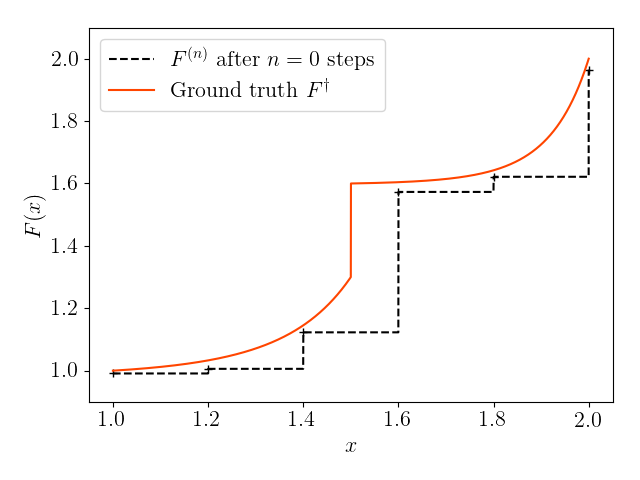} }
	\subfigure[$n = 10$]{\includegraphics[width=0.32\textwidth]{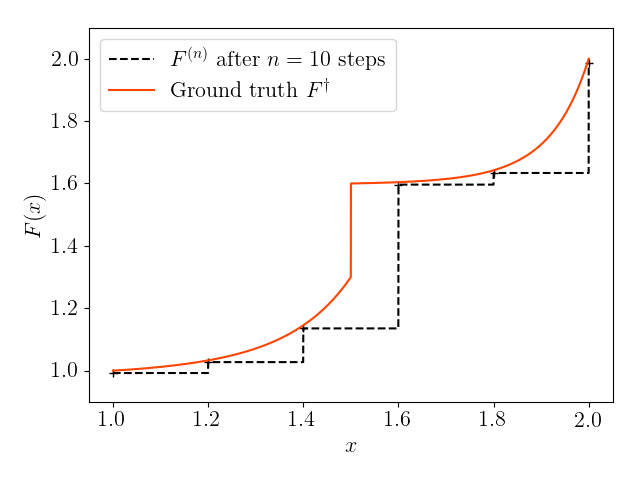} }
	\subfigure[$n = 50$]{\includegraphics[width=0.32\textwidth]{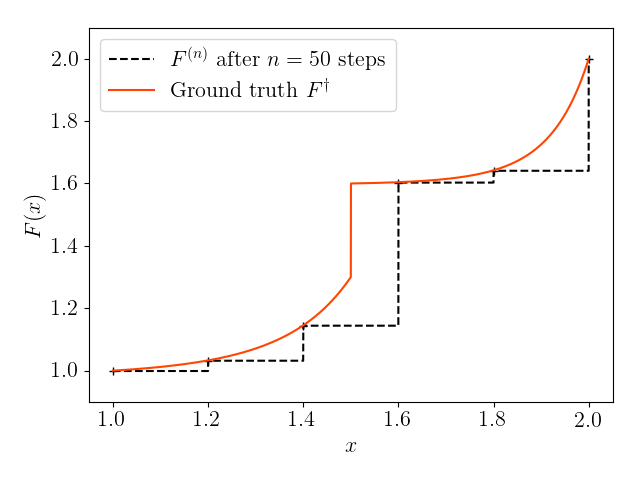} }
	\subfigure[$n = 100$]{\includegraphics[width=0.32\textwidth]{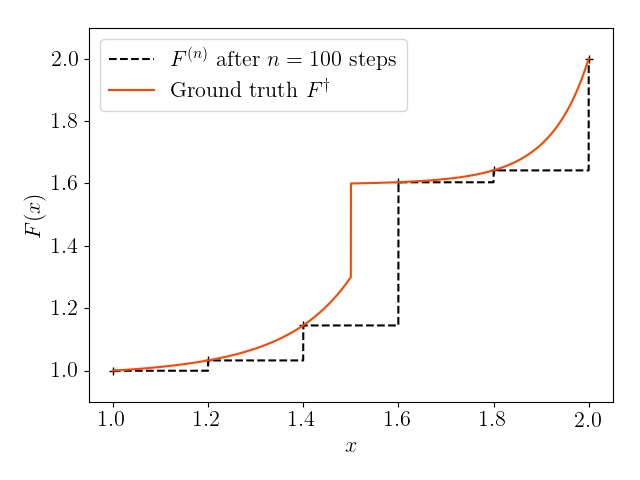} }
	\subfigure[Minimum of the quality]{\includegraphics[width=0.32\textwidth]{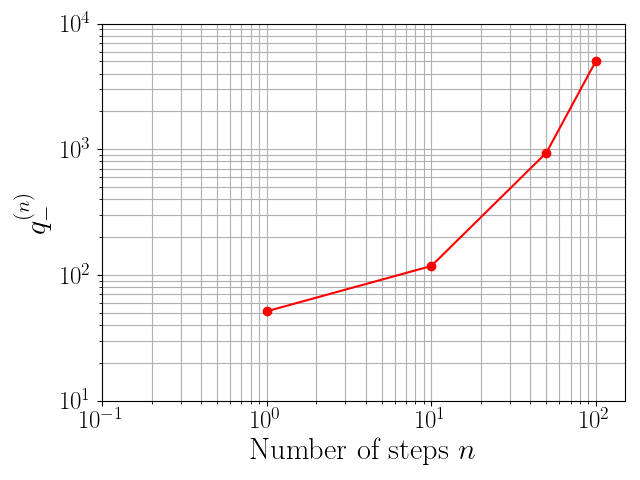}\label{fig:minqual_HIGHER} }
	\subfigure[Total area]{\includegraphics[width=0.32\textwidth]{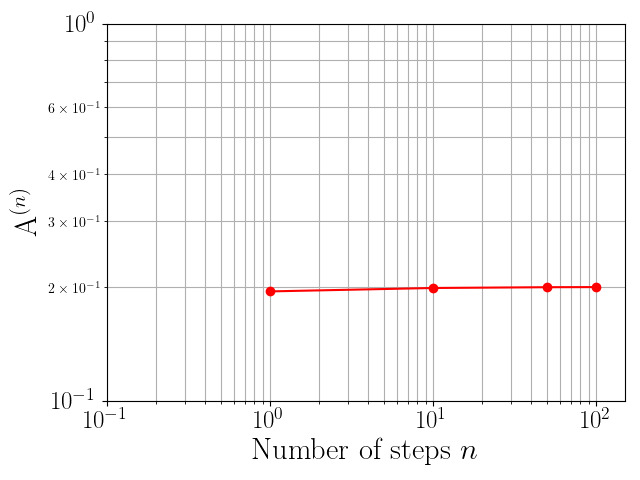}\label{fig:sumarea_HIGHER} }
	\caption{Evolution of $\step{\Fapp}{n}$ and the minimum of the quality and the total area as functions of the iteration count, $n$, for a discontinuous ground truth $\Fdag$ with $\ER = 10^{4}$.}
	\label{fig:Cont_ER_HIGHER}
\end{figure}

\section{Application to optimal uncertainty quantification}%
\label{sec:application_ouq}%

\subsection{Optimal uncertainty quantification}
\label{sec:ouq_overview}

In the \emph{optimal uncertainty quantification} paradigm proposed by \citet{Owhadi2013} and further developed by, e.g., \citet{Sullivan2013} and \citet{Han2015}, upper and lower bounds on the performance of an incompletely-specified system are calculated via optimisation problems.
More concretely, one is interested in the probability that a system, whose output is a function $g^{\dagger} \colon \mathcal{X} \to \Reals$ of inputs $\mathbf{\Xi}$ distributed according to a probability measure $\mu^{\dagger}$ on an input space $\mathcal{X}$, satisfies $g^{\dagger} (\mathbf{\Xi}) \leq x$, where $x$ is a specified performance threshold value.
We emphasise that although we focus on a scalar performance measure, the input $\mathbf{\Xi}$ may be a multivariate random variable.

In practice, $\mu^{\dagger}$ and $g^{\dagger}$ are not known exactly;
rather, it is known only that $(\mu^{\dagger}, g^{\dagger}) \in \mathcal{A}$ for some admissible subset $\mathcal{A}$ of the product space of all probability measures on $\mathcal{X}$ with the set of all real-valued functions on $\mathcal{X}$.
Thus, one is interested in
\[
	\underline{P}_{\mathcal{A}} (x) \defeq \inf_{(\mu, g) \in \mathcal{A}} \P_{\mathbf{\Xi} \sim \mu} [ g(\mathbf{\Xi}) \leq x ]
	\quad
	\text{and}
	\quad
	\overline{P}_{\mathcal{A}} (x) \defeq \sup_{(\mu, g) \in \mathcal{A}} \P_{\mathbf{\Xi} \sim \mu} [ g(\mathbf{\Xi}) \leq x ] .
\]
The inequality
\[
	0 \leq \underline{P}_{\mathcal{A}} (x) \leq \P_{\mathbf{\Xi} \sim \mu^{\dagger}} [ g^{\dagger}(\mathbf{\Xi}) \leq x ] \leq \overline{P}_{\mathcal{A}} (x) \leq 1
\]
is, by definition, the tightest possible bound on the quantity of interest $\P_{\mathbf{\Xi} \sim \mu^{\dagger}} [ g^{\dagger}(\mathbf{\Xi}) \leq x ]$ that is compatible with the information used to specify $\mathcal{A}$.
Thus, the optimal UQ perspective enriches the principles of worst- and best-case design to account for distributional and functional uncertainty.
We concentrate our attention hereafter, without loss of generality, on the least upper bound $\overline{P}_{\mathcal{A}} (x)$.

\begin{Remark}
	\label{rmk:reduction}
	The main focus of this paper is the dependency of $\overline{P}_{\mathcal{A}} (x)$ on $x$.
	In practice, an underlying task is, for any individual $x$, reducing the calculation of $\overline{P}_{\mathcal{A}} (x)$ to a tractable finite-dimensional optimisation problem.
	Central enabling results here are the \emph{reduction theorems} of \citet[Section~4]{Owhadi2013}, which, loosely speaking, say that if, for each $g$, $\{ \mu \mid (\mu, g) \in \mathcal{A} \}$ is specified by a system of $m$ equality or inequality constraints on expected values of arbitrary test functions under $\mu$, then for the determination of $\overline{P}_{\mathcal{A}} (x)$ it is sufficient to consider only distributions $\mu$ that are convex combinations of at most $m + 1$ point masses;
	the optimisation variables are then the $m$ independent weights and $m + 1$ locations in $\mathcal{X}$ of these point masses.
	If $\mu$ factors as a product of distributions (i.e.\ $\mathbf{\Xi}$ is a vector with independent components), then this reduction theorem applies componentwise.
\end{Remark}

As \revised{a function} of the performance threshold $x$, $\overline{P}_{\mathcal{A}} (x)$ \revised{is} an increasing function, and so it is potentially advantageous to determine $\overline{P}_{\mathcal{A}} (x)$ jointly for a wide range of $x$ values using the algorithm developed above.
Indeed, determining $\overline{P}_{\mathcal{A}} (x)$ for many values of $x$, rather than just one value, is desirable for multiple reasons:
\begin{enumerate}
	\item Since numerical optimisation to determine $\overline{P}_{\mathcal{A}} (x)$ may be affected by errors, computing several values of $\overline{P}_{\mathcal{A}} (x)$ could lead to validate their consistency as the function $x \mapsto \overline{P}_{\mathcal{A}} (x)$ must be increasing;
	\item The function $\overline{P}_{\mathcal{A}} (x)$ can be discontinuous.
	Thus, by computing several values of $\overline{P}_{\mathcal{A}} (x)$, one can highlight potential discontinuities and can identify key threshold values of $x \mapsto \overline{P}_{\mathcal{A}} (x)$.
\end{enumerate}

\subsection{Test case}

For the application of \Cref{our_algorithm} to OUQ, we study the robust shape optimization of the two-dimensional RAE2822 airfoil \citep[Appendix A6]{Cook1979} using ONERA's CFD software \textit{elsA} \citep{Cambier2008}. The following example is taken from \citet{Dumont2019}.
The shape of the original RAE2822 is altered using four bumps located at four different locations: $5\%$, $20\%$, $40\%$, and $60\%$ of the way along the chord $c$ (see \Cref{fig:BumpAirfoil}).
These bumps are characterised by B-splines functions. 

\begin{figure}[h]
	\centering
	\includegraphics[scale = 0.3]{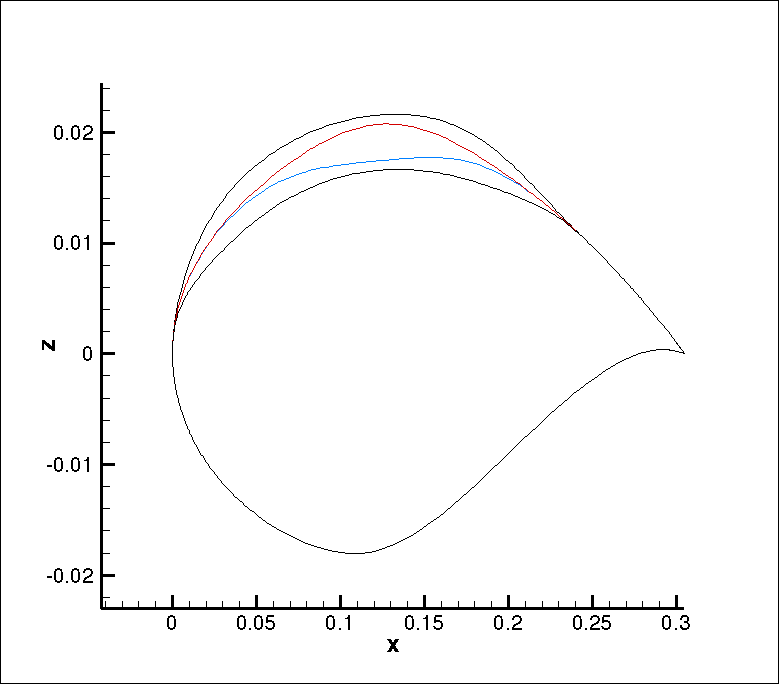}
	\caption{Black lines: Maximum and minimum deformation of the RAE2822 profile.
	Red: Maximum deformation of the third bump alone.
	Blue: Minimum deformation of the third bump alone.
	This image is taken from \citet{Dumont2019}.}
	\label{fig:BumpAirfoil}
\end{figure}

The lift-to-drag ratio $\frac{C_l}{C_d}$ of the RAE2822 wing profile (see \Cref{fig:airfoil}) at Reynolds Number $Re = 6.5\cdot10^{6}$, Mach number $M_{\infty} = 0.729$ and angle of attack $\alpha = 2.31\degree$ is chosen as the performance function $g^{\dagger}$ with inputs $\mathbf{\Xi} = (\Xi_1, \Xi_2, \Xi_3, \Xi_4)$, where $(\Xi_i)_{i=1\dots4}$ is the amplitude of each bump.
They will be considered as random variables over their respective range given in \Cref{tab:ISES_inputs}.

\begin{figure}
	\centering
	\includegraphics[scale = 0.7]{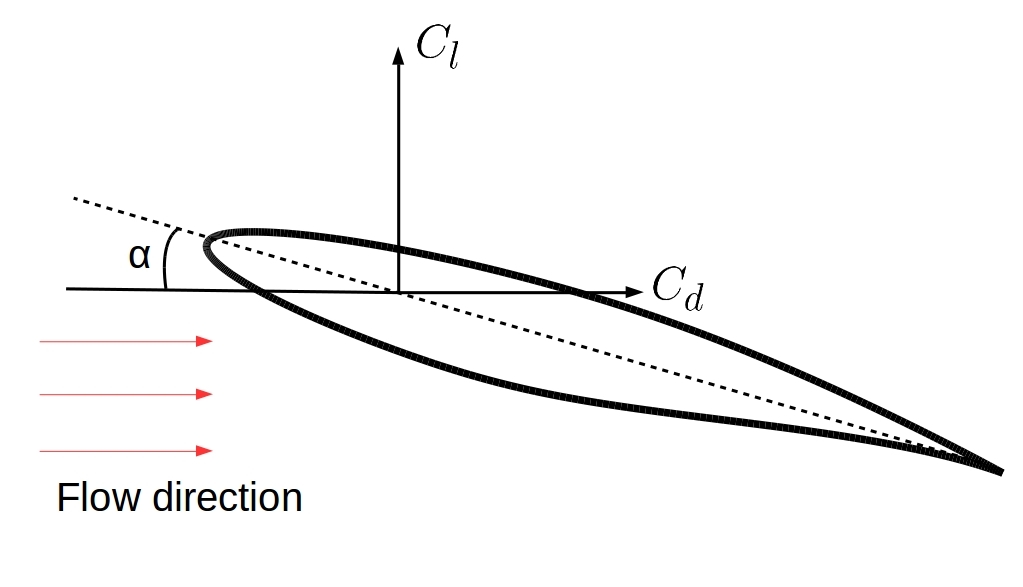}
	\caption{Picture depicting the lift $C_l$ and the drag $C_d$ of an airfoil.}
	\label{fig:airfoil}
\end{figure}

\begin{table}[h]
	\centering
	\begin{tabular}{| c | c | c |}
		\hline
		& Range & Law \\ \hline
		Bump 1: $\Xi_1$ & [-0.0025c; +0.0025c] & $\mu^{\dagger}_1$: Beta law with $\alpha = 6, \beta = 6$\\ \hline
		Bump 2: $\Xi_2$ & [-0.0025c; +0.0025c] & $\mu^{\dagger}_2$: Beta law with $\alpha = 2, \beta = 2$ \\ \hline
		Bump 3: $\Xi_3$ & [-0.0025c; +0.0025c] & $\mu^{\dagger}_3$: Beta law with $\alpha = 2, \beta = 2$ \\ \hline
		Bump 4: $\Xi_4$ & [-0.0025c; +0.0025c] & $\mu^{\dagger}_4$: Beta law with $\alpha = 2, \beta = 2$ \\
		\hline
	\end{tabular}
	\caption{Range of each input parameter.}
	\label{tab:ISES_inputs}
\end{table}

The corresponding flow values are the ones described in test case $\#6$ together with the wall interferences corrections formulas given in \citet[Chapter 6]{Garner1966} and in \citet[Section 5.1]{Haase1993}.
Moreover, we will assume that $(\Xi_i)_{i=1\dots4}$ are mutually independent.
An ordinary Kriging procedure has been chosen to build a metamodel (or response surface) of $g^{\dagger}$, which is identified with the actual response function $g^\dag$ in the subsequent analysis.
A tensorised grid of 9 equidistributed abscissas for each parameter is used.
The model is then based on $N = 9^4 = 6561$ observations. 
In that respect, a Gaussian kernel
\begin{equation*}
	K(\mathbf{\Xi},\mathbf{\Xi'})= \exp \left( -\frac{1}{2}\sum_{i=1}^{4}\frac{(\Xi_i - \Xi'_i)^{2}}{\gamma_i^2} \right)
\end{equation*}
has been chosen, where $\mathbf{\Xi} = (\Xi_1, \Xi_2, \Xi_3, \Xi_4)$ and $\mathbf{\Xi'} = (\Xi'_1, \Xi'_2, \Xi'_3, \Xi'_4)$ are inputs of the function $g^{\dagger}$, and where $\gamma = (\gamma_1, \gamma_2, \gamma_3, \gamma_4)$ are the parameters of the kernel.
These parameters are chosen to minimize the variance between the ground truth data defined by the $N$ observations and their Kriging metamodel $g^{\dagger}$.
The responce surfaces in the $(\Xi_1, \Xi_3)$ plan for two values of $(\Xi_2, \Xi_4)$ are shown on \Cref{fig:surface_init}.

\begin{figure}[t]
	\centering
	\subfigure[$\Xi_2 = -0.0025,\;\Xi_4 = 0$.]{\includegraphics[width=0.45\textwidth]{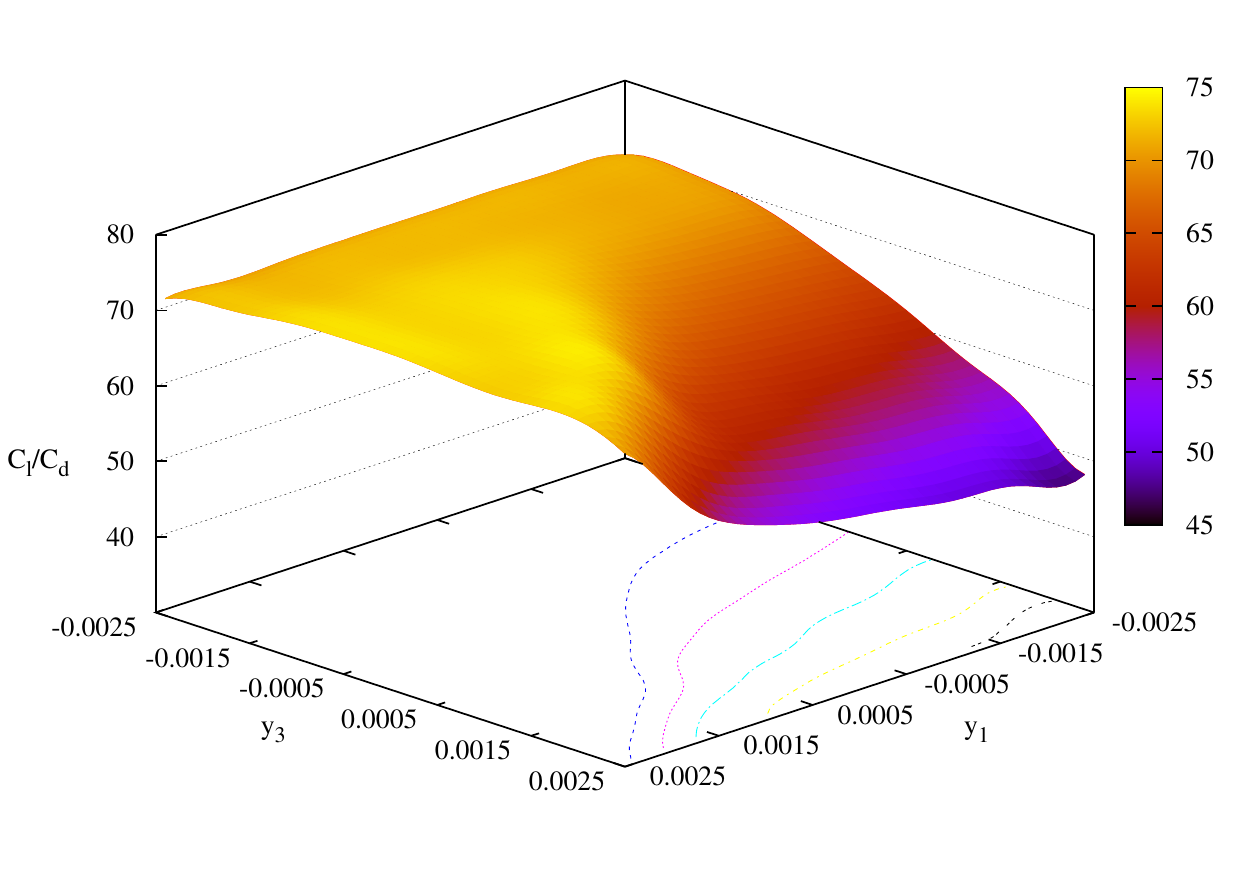} }\label{subfig:surf_init_a}
	\subfigure[$\Xi_2 = 0.0025,\;\Xi_4 = 0$.]{\includegraphics[width=0.45\textwidth]{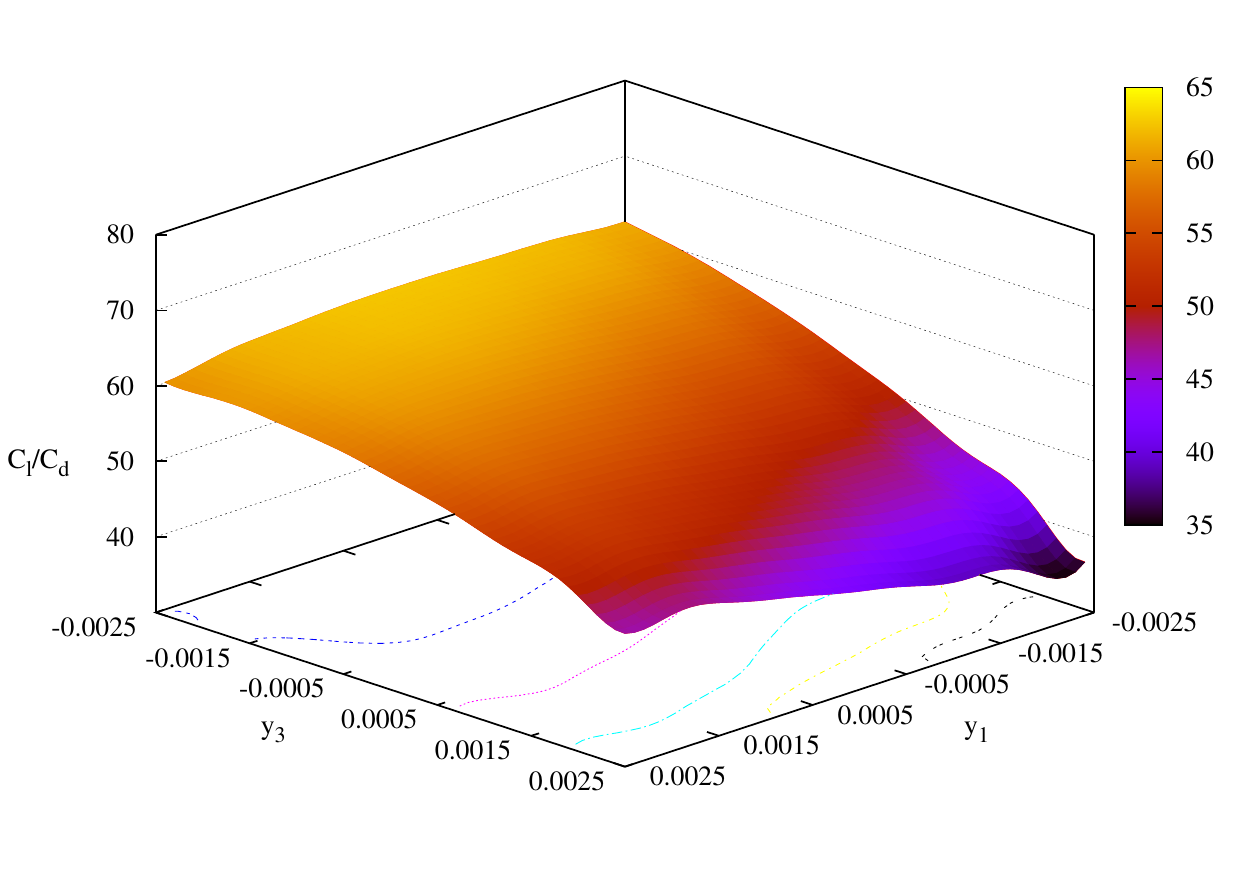} }\label{subfig:surf_init_b}
	\caption{Response surface in the $(\Xi_1, \Xi_3)$ plane with $(\Xi_2 = -0.0025, \Xi_4= 0)$ (a) and $(\Xi_2 = 0.0025, \Xi_4= 0)$ (b).
	These images are taken from \citet{Dumont2019}.}
	\label{fig:surface_init}
\end{figure}

One seeks to determine $\overline{P}_{\mathcal{A}} (x) \defeq \sup_{\mu \in \mathcal{A}} \P_{\mathbf{\Xi} \sim \mu} [ g^{\dagger}(\mathbf{\Xi}) \leq x ]$, where the admissible set $\mathcal{A}$ is defined as follows:
\begin{equation}
	\label{eq:A_test}
	\mathcal{A} = \left\{ (g, \mu) \,\middle|\, \begin{array}{c}
	\mathbf{\Xi} \in \mathcal{X} = \mathcal{X}_1 \times \mathcal{X}_2 \times \mathcal{X}_3 \times \mathcal{X}_4 \\
	g: \mathcal{X} \mapsto \mathcal{Y} \text{\;is known equal to\;} g^{\dagger}\\
	\mu = \mu_1 \otimes \mu_2 \otimes \mu_3\otimes \mu_4\\
	\mathbb{E}_{\mathbf{\Xi}\sim \mu}[g(\mathbf{\Xi})] = \text{LD}
	\end{array}\right\}.
\end{equation}
A priori, finding $\overline{P}_{\mathcal{A}} (x)$ is not computationally tractable because it requires a search over a infinite-dimensional space of probability measures defined by $\mathcal{A}$.
Nevertheless, as described briefly in \Cref{rmk:reduction}, it has been shown in \citet{Owhadi2013} that this optimisation problem can be reduced to a finite-dimensional one, where now the probability measures are products of finite convex combinations of Dirac masses.

\begin{Remark}
	The ground truth law $\mu^\dag$ of each input variable given in \Cref{tab:ISES_inputs} is only used to compute the expected value $\mathbb{E}_{\mathbf{\Xi}\sim \mu}[g(\mathbf{\Xi})] = \text{LD}$.
	This expected value is computed with $10^4$ samples.
\end{Remark}

\begin{Remark}
	The admissible set $\mathcal{A}$ from \eqref{eq:A_test} can be understood as follows: 
	\begin{itemize}
		\item One knows the range of each input parameter $(\Xi_i)_{i=1,\dots,4}$;
		\item $g$ is exactly known as $g = g^{\dagger}$;
		\item $(\Xi_i)_{i=1,\dots,4}$ are independent;
		\item One only knows the expected value of $g$: $\mathbb{E}_{\mathbf{\Xi}\sim \mu}[g(\mathbf{\Xi})]$.
	\end{itemize} 
\end{Remark}

The optimisation problem of determining $\overline{P}_{\mathcal{A}} (x)$ for each chosen $x$ was solved using the Differential Evolution algorithm of \citet{Storn97} within the \textit{mystic} optimisation framework \citep{McKerns2011}.
Ten iterations of \Cref{our_algorithm} have been performed using $\ER = 1 \times 10^{4}$.
The evolution of $\overline{P}_{\mathcal{A}} (x)$ as function of the iteration count, $n$, is shown on \Cref{fig:OUQ_ALGO}.
At $n=0$ --- see \Cref{fig:OUQ_ALGO}\subref{fig:OUQ_0} --- two consistent points are present at $x = 57.51$ and $x = 67.51$.
At this step, $\step{\WA}{0} = 35289$. As $\step{\WA}{0} \geq \ER$, at next step $n=1$, the algorithm adds a new point at the middle of the biggest rectangle --- see \Cref{fig:OUQ_ALGO}\subref{fig:OUQ_1} and \Cref{fig:OUQ_q_area}\subref{fig:total_area_OUQ}.
After $n=10$ steps, eight points are now present in total with a minimum quality increasing from $5000$ to $11667$ and with a total area decreasing from $7.05$ to $0.84$;
see \Cref{fig:OUQ_q_area}\subref{fig:min_q_OUQ} and \Cref{fig:OUQ_q_area}\subref{fig:total_area_OUQ} respectively.

\begin{figure}[t]
	\centering
	\subfigure[$n = 0$]{\includegraphics[width=0.40\textwidth]{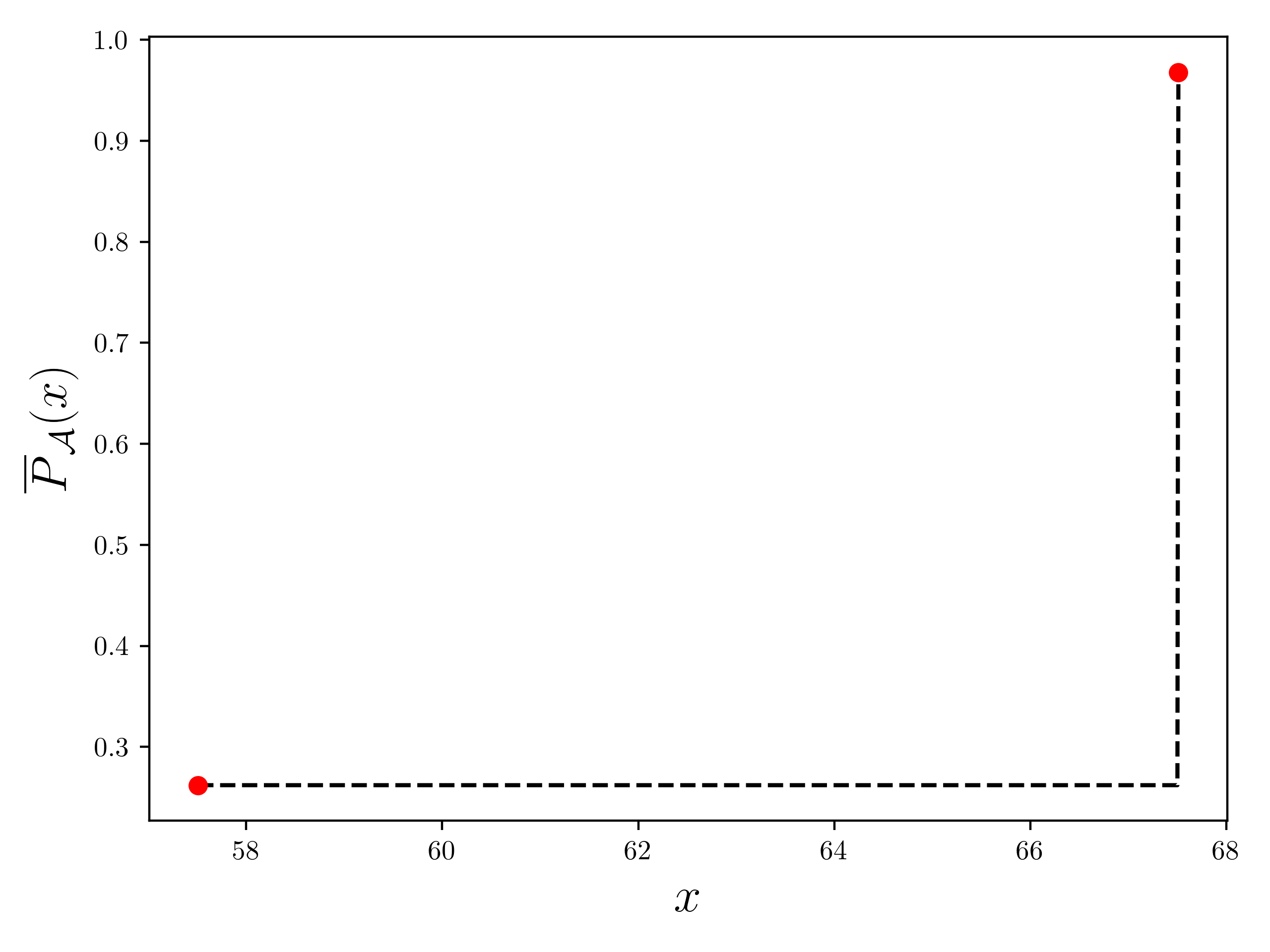} \label{fig:OUQ_0}}
	\subfigure[$n = 1$]{\includegraphics[width=0.40\textwidth]{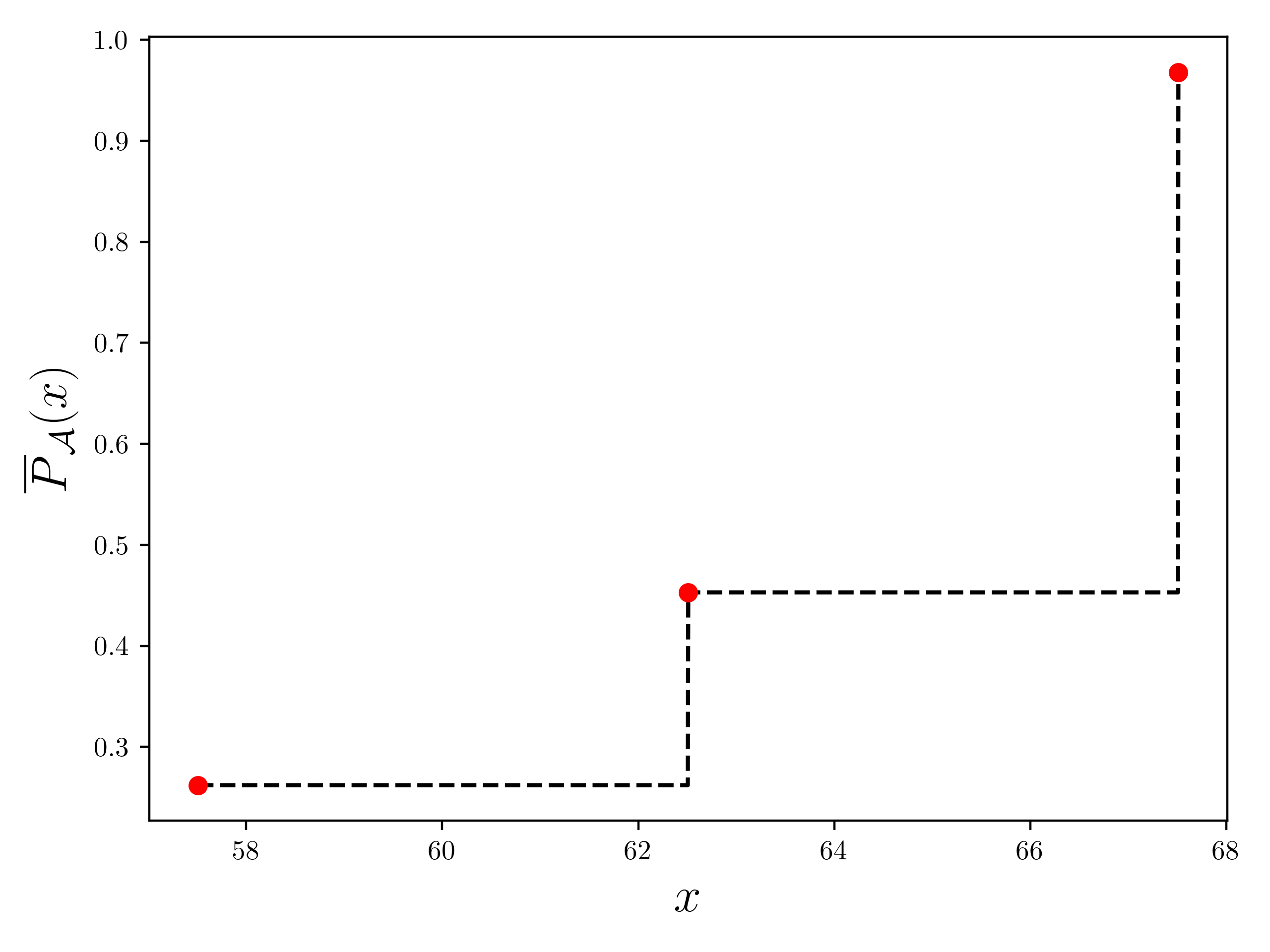} \label{fig:OUQ_1}} \hfill
	\subfigure[$n = 5$]{\includegraphics[width=0.40\textwidth]{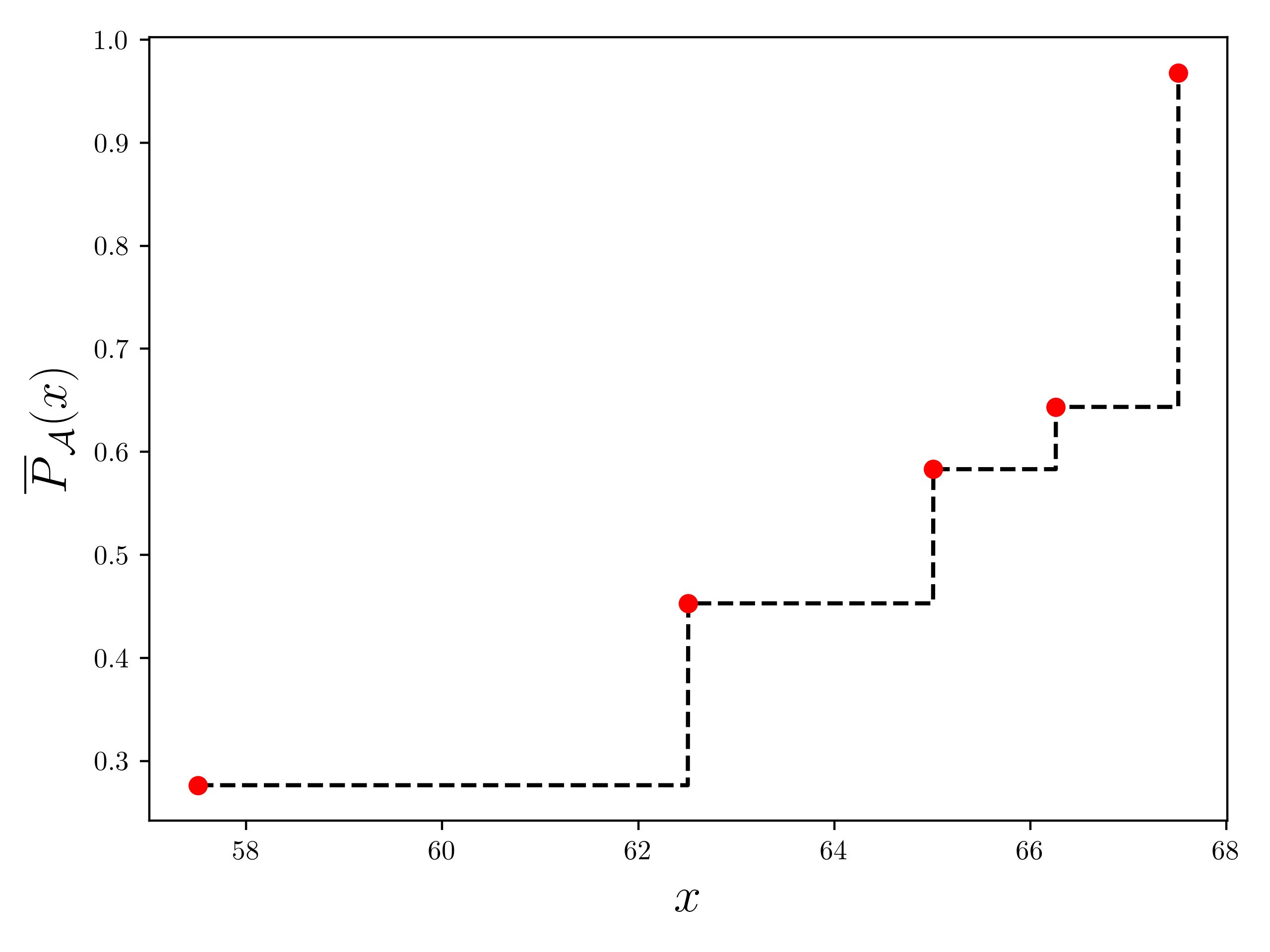} \label{fig:OUQ_5}}
	\subfigure[$n = 10$]{\includegraphics[width=0.40\textwidth]{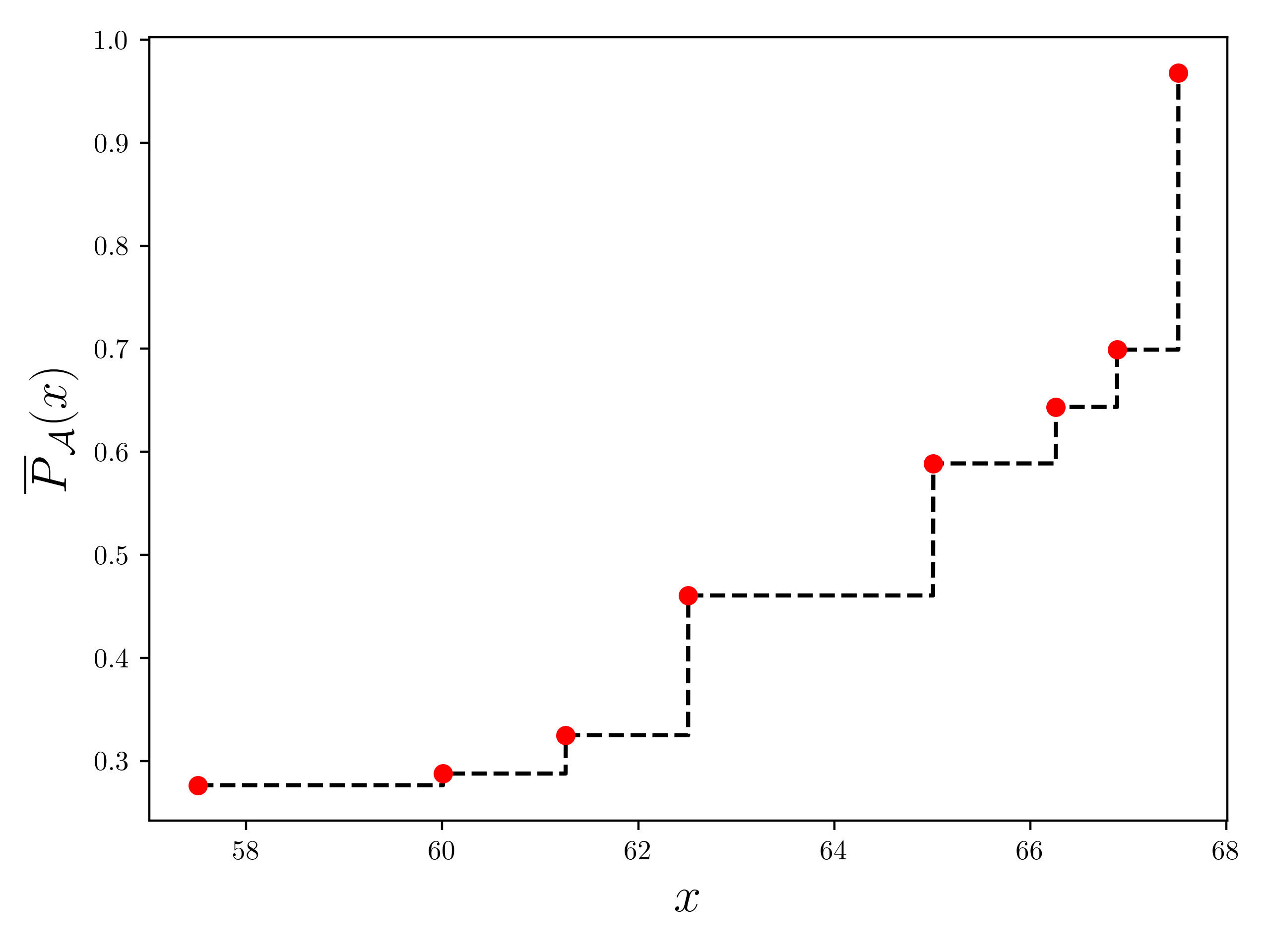} \label{fig:OUQ_10}}
	\caption{Evolution of $\overline{P}_{\mathcal{A}} (x)$ as function of the iteration count, $n$.}
	\label{fig:OUQ_ALGO}
\end{figure}

\begin{figure}[ht!]
	\centering
	\subfigure[Evolution of the minimum of the quality.]{\includegraphics[width=0.40\textwidth]{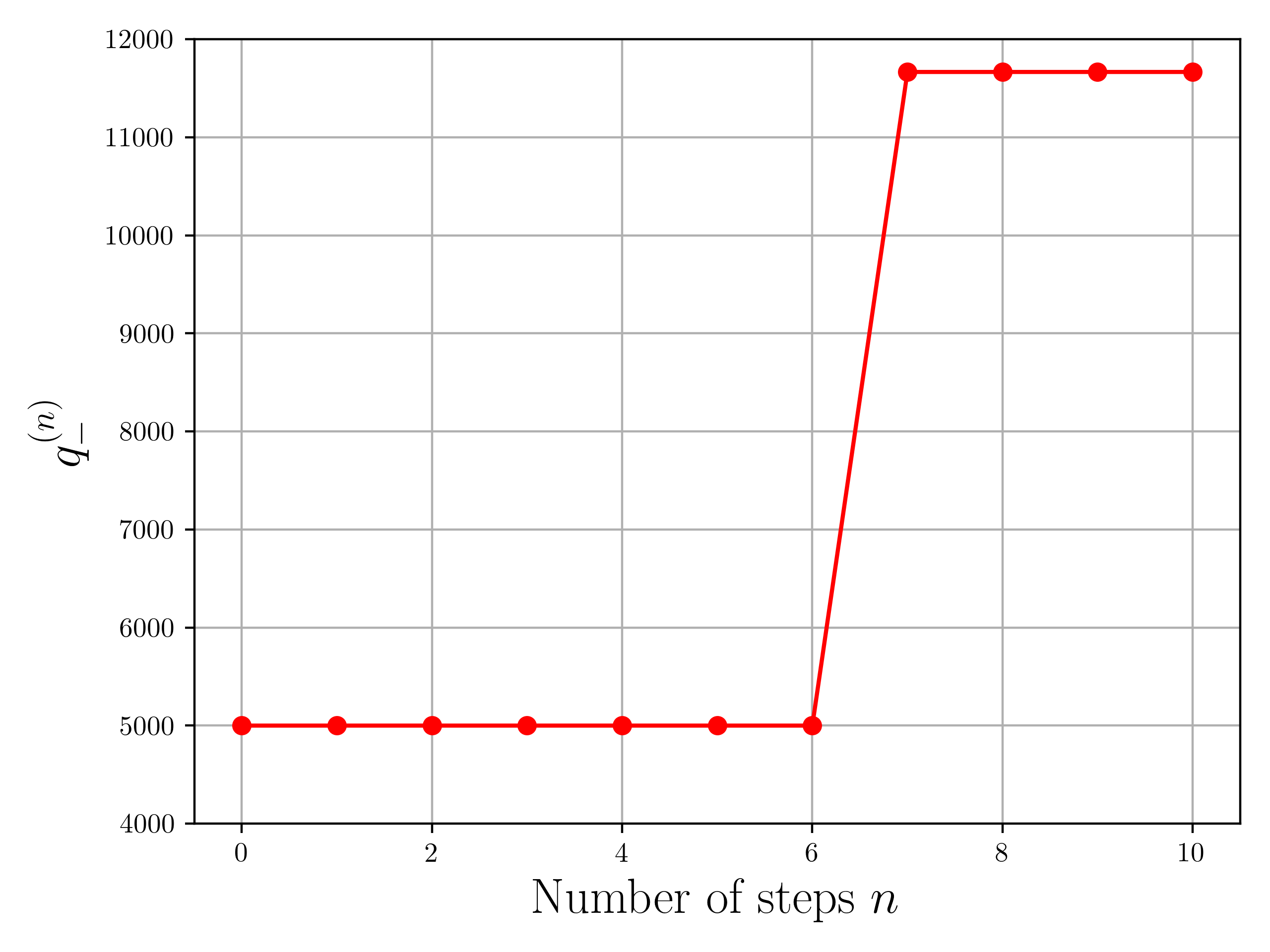} \label{fig:min_q_OUQ}}
	\subfigure[Evolution of the total area.]{\includegraphics[width=0.40\textwidth]{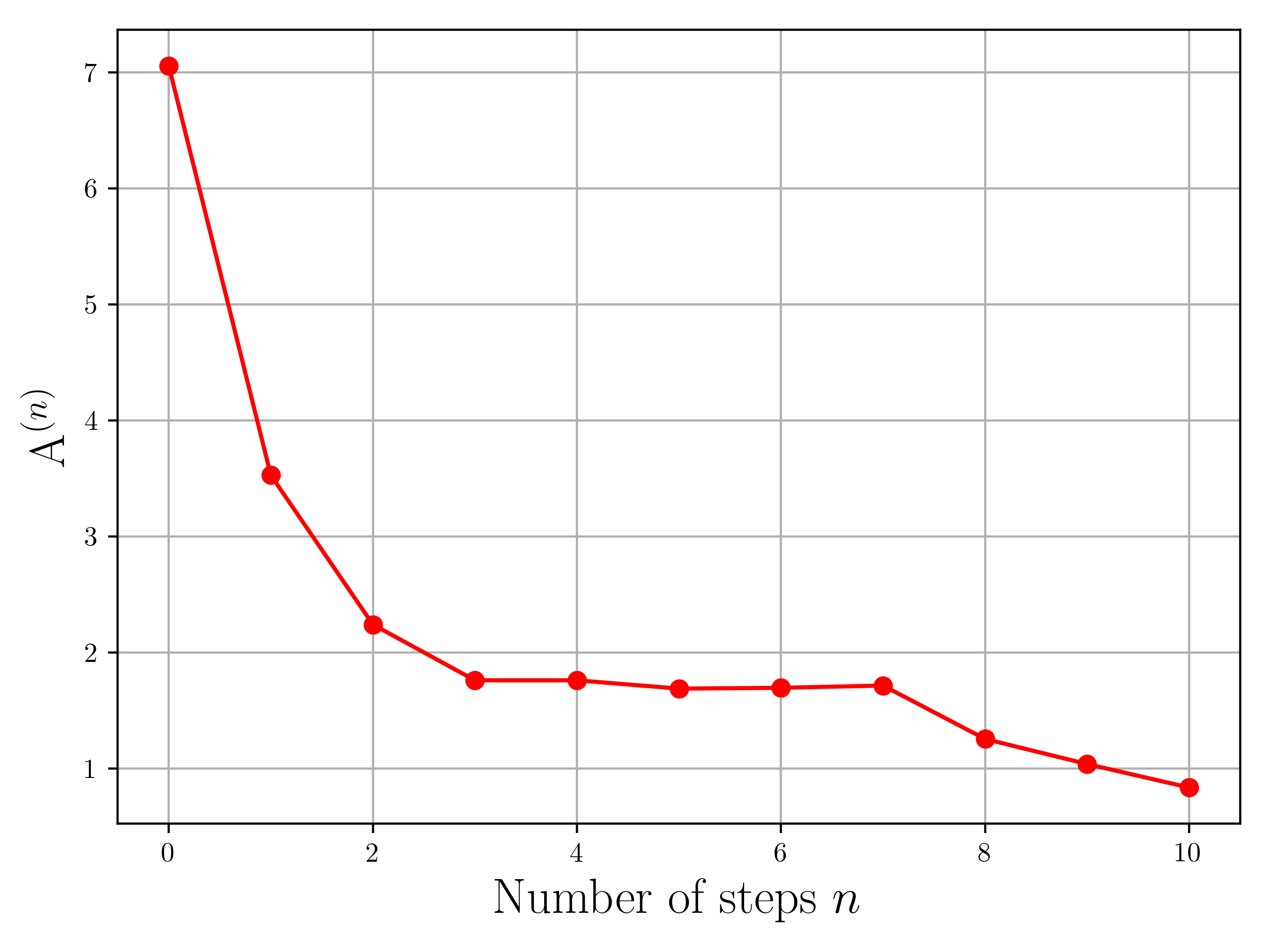}
	\label{fig:total_area_OUQ}}
	\caption{Evolution of the minimum of the quality and the total area as function of the iteration count, $n$.}
	\label{fig:OUQ_q_area}
\end{figure}

\revised{The number of iterations in this complex numerical experiment has been limited to $10$ because obtaining new or improved data points consistent throughout the optimization algorithm may take up to two days (wall-clock time on a personal computer equipped with an Intel Core i5-6300HQ processor with $4$ cores and $6$\,MB cache memory) for one single point. This running time is increased further for data points of higher quality. Nevertheless, this experiment shows that the proposed algorithm can be used for real-world examples in an industrial context}.


\section{Concluding remarks}%
\label{sec:conclusion}%

In this paper we have developed an algorithm to reconstruct a monotonically increasing function such as the cumulative distribution function of a real-valued random variable, or the least upper bound of the performance criterion of a system as a function of its performance threshold.
In particular, this latter setting has relevance to the optimal uncertainty quantification (OUQ) framework of \cite{Owhadi2013} we have in mind for applications to real-world incompletely specified systems.
The algorithm uses imperfect pointwise evaluations of the target function, subject to partially controllable one-sided errors, to direct further evaluations either at new sites in the function's domain or to improve the quality of evaluations at already-evaluated sites.
It allows for some flexibility at targeting either strategy through a user-defined ``exchange rate'' parameter, yielding an approximation of the target function with few high-quality points or alternatively more lower-quality points.
We have studied its convergence properties and have applied it to several examples:
known target functions that are either continuous and discontinuous, and a performance function for aerodynamic design of a well-documented standard profile in the OUQ setting.

\Cref{our_algorithm} is reminiscent of the classical PAVA approach to isotonic regression that applies to statistical inference with order restrictions.
Examples of its use can be found in shape constrained or parametric density problems as illustrated in e.g.\ \cite{Groeneboom2014}.
Possible improvements and extensions of our algorithm include weighting the areas $\smash{\step{\area_i}{n}}$ as they are summed up to form the total weighted area $\smash{\step{\WA}{n}}$ driving the iterative process, in order to optimally enforce both the addition of ``steps'' $\smash{\step{s_i}{n}}$ in the reconstruction function $\smash{\step{\Fapp}{n}}$ of \Cref{def:piewise_constant_funct}, and the improvement of their ``heights'' $\smash{\step{\yobs_i}{n}}$.
This could be achieved considering for example the following alternative definition $\smash{\step{i_+}{n}}=\smash{\arg\max_i\{(\step{\Nobs}{n}-i-1)\step{\area_i}{n}\}}$ in \Cref{our_algorithm}, which results in both adding a step to the $\smash{\step{i_+}{n}}$-th current one and possibly improving all subsequent evaluations $\smash{\step{\yobs_i}{n+1}}$, $i>\smash{\step{i_+}{n}}$.
We may further envisage to adapt the ideas elaborated in this research to the reconstruction of convex functions by extending the notion of consistency.
These perspectives shall be considered in future works.






\section*{Acknowledgements}
\addcontentsline{toc}{section}{Acknowledgements}

The work of J.L.A.\ and {\'E}.S.\ has been partially supported by ONERA within the Laboratoire de Math\-\'ema\-tiques Appliqu\'ees pour l'A\'eronautique et Spatial (LMA$^2$S).
L.B.\ is supported by a CDSN grant from the French Ministry of Higher Education (MESRI) and a grant from the German Academic Exchange Service (DAAD), Program \#57442045.
T.J.S.\ has been partially supported by the Freie Universit\"at Berlin within the Excellence Strategy of the DFG, including project TrU-2 of the Excellence Cluster ``MATH+ The Berlin Mathematics Research Center'' (EXC-2046/1, project 390685689) and DFG project 415980428.

\bibliographystyle{abbrvnat}
\addcontentsline{toc}{section}{References}

\bibliography{references}

\begin{thebibliography}{16}
\providecommand{\natexlab}[1]{#1}
\providecommand{\url}[1]{\texttt{#1}}
\expandafter\ifx\csname urlstyle\endcsname\relax
  \providecommand{\doi}[1]{doi: #1}\else
  \providecommand{\doi}{doi: \begingroup \urlstyle{rm}\Url}\fi

\bibitem[Barlow et~al.(1972)Barlow, Bartholomew, Bremner, and
  Brunk]{Barlow1972}
R.~E. Barlow, D.~J. Bartholomew, J.~M. Bremner, and H.~D. Brunk.
\newblock \emph{Statistical {I}nterference under {O}rder {R}estrictions. {T}he
  {T}heory and {A}pplication of {I}sotonic {R}egression}.
\newblock John Wiley \& Sons, London-New York-Sydney, 1972.

\bibitem[Cambier et~al.(2013)Cambier, Heib, and Plot]{Cambier2008}
L.~Cambier, S.~Heib, and S.~Plot.
\newblock The {ONERA} els{A} {CFD} software: input from research and feedback
  from industry.
\newblock \emph{Mechanics \& Industry}, 14\penalty0 (3):\penalty0 159--174,
  2013.
\newblock \doi{10.1051/meca/2013056}.

\bibitem[Cook et~al.(1979)Cook, McDonald, and Firmin]{Cook1979}
P.~H. Cook, M.~A. McDonald, and M.~C.~P. Firmin.
\newblock Aerofoil {RAE} 2822 pressure distributions, and boundary layer and
  wake measurements.
\newblock In \emph{Experimental data base for computer program assessment.
  {AGARD} Advisory Report No. 138}. NATO, 1979.
\newblock URL \url{http://eda-ltd.com.tr/caeeda\_doc/AGARD-AR-138.pdf}.

\bibitem[de~Leeuw et~al.(2009)de~Leeuw, Hornik, and Mair]{Jan2009}
J.~de~Leeuw, K.~Hornik, and P.~Mair.
\newblock Isotone optimization in {R}: {P}ool-{A}djacent-{V}iolators
  {A}lgorithm ({PAVA}) and active set methods.
\newblock \emph{J. Stat. Software}, 32\penalty0 (5):\penalty0 1--24, 2009.
\newblock \doi{10.18637/jss.v032.i05}.

\bibitem[Dumont et~al.(2019)Dumont, Hantrais-Gervois, Passaggia, Peter,
  Salah~el Din, and Savin]{Dumont2019}
A.~Dumont, J.-L. Hantrais-Gervois, P.-Y. Passaggia, J.~Peter, I.~Salah~el Din,
  and {\'E}.~Savin.
\newblock Ordinary kriging surrogates in aerodynamics.
\newblock In C.~Hirsch, D.~Wunsch, J.~Szumbarski,
  {\L}.~{\L}aniewski-Wo{\l}{\l}k, and J.~Pons-Prats, editors, \emph{Uncertainty
  Management for Robust Industrial Design in Aeronautics: Findings and Best
  Practice Collected During UMRIDA, a Collaborative Research Project
  (2013--2016) Funded by the European Union}, pages 229--245. Springer
  International Publishing, Cham, 2019.
\newblock \doi{10.1007/978-3-319-77767-2_14}.

\bibitem[Garner et~al.(1966)Garner, Rogers, Acum, and Maskell]{Garner1966}
H.~C. Garner, E.~W.~E. Rogers, W.~E.~A. Acum, and E.~C. Maskell.
\newblock Subsonic wind tunnel jwall corrections.
\newblock {AGARD}o-graph 109, NATO, 1966.

\bibitem[Groeneboom and Jongbloed(2014)]{Groeneboom2014}
P.~Groeneboom and G.~Jongbloed.
\newblock \emph{Nonparametric {E}stimation under {S}hape {C}onstraints:
  {E}stimators. Algorithms and {A}symptotics}.
\newblock Cambridge University Press, Cambridge, 2014.
\newblock \doi{10.1017/CBO9781139020893}.

\bibitem[Haase et~al.(1993)Haase, Bradsma, Elsholz, Leschziner, and
  Schwamborn]{Haase1993}
W.~Haase, F.~Bradsma, E.~Elsholz, M.~Leschziner, and D.~Schwamborn.
\newblock \emph{{EUROVAL} - {A}n {E}uropean {I}nitiative on {V}alidation of
  {CFD} {C}odes}.
\newblock Vieweg Verlag, Wiesbaden, 1993.
\newblock \doi{10.1007/978-3-663-14131-0}.

\bibitem[Han et~al.(2015)Han, Tao, Topcu, Owhadi, and Murray]{Han2015}
S.~Han, M.~Tao, U.~Topcu, H.~Owhadi, and R.~M. Murray.
\newblock Convex {O}ptimal {U}ncertainty {Q}uantification.
\newblock \emph{SIAM J. Optim.}, 25\penalty0 (3):\penalty0 1368--1387, 2015.
\newblock \doi{10.1137/13094712X}.

\bibitem[Jordan et~al.(2019)Jordan, Mühlemann, and Ziegel]{Jordan2019}
A.~I. Jordan, A.~Mühlemann, and J.~F. Ziegel.
\newblock Optimal solutions to the isotonic regression problem, 2019.
\newblock URL \url{https://arxiv.org/abs/1904.04761}.

\bibitem[McKerns et~al.(2011)McKerns, Strand, Sullivan, Fang, and
  Aivazis]{McKerns2011}
M.~M. McKerns, L.~Strand, T.~J. Sullivan, A.~Fang, and M.~A.~G. Aivazis.
\newblock {Building a framework for predictive science}.
\newblock In S.~van~der Walt and J.~Millman, editors, \emph{Proceedings of the
  10th Python in Science Conference (SciPy 2011), June 2011}, pages 67--78,
  2011.
\newblock \doi{10.25080/Majora-ebaa42b7-00d}.

\bibitem[Owhadi et~al.(2013)Owhadi, Scovel, Sullivan, McKerns, and
  Ortiz]{Owhadi2013}
H.~Owhadi, C.~Scovel, T.~J. Sullivan, M.~McKerns, and M.~Ortiz.
\newblock Optimal {U}ncertainty {Q}uantification.
\newblock \emph{SIAM Rev.}, 55\penalty0 (2):\penalty0 271--345, 2013.
\newblock \doi{10.1137/10080782X}.

\bibitem[Robertson et~al.(1988)Robertson, Wright, and Dykstra]{Robertson1988}
T.~Robertson, F.~T. Wright, and R.~L. Dykstra.
\newblock \emph{Order {R}estricted {S}tatistical {I}nference}.
\newblock John Wiley \& Sons, London-New York-Sydney, 1988.

\bibitem[Storn and Price(1997)]{Storn97}
R.~Storn and K.~Price.
\newblock Differential evolution -- a simple and efficient heuristic for global
  optimization over continuous spaces.
\newblock \emph{Journal of Global Optimization}, 11:\penalty0 341--359, 1997.
\newblock \doi{10.1023/A:1008202821328}.

\bibitem[Sullivan et~al.(2013)Sullivan, McKerns, Meyer, Theil, Owhadi, and
  Ortiz]{Sullivan2013}
T.~J. Sullivan, M.~McKerns, D.~Meyer, F.~Theil, H.~Owhadi, and M.~Ortiz.
\newblock Optimal {U}ncertainty {Q}uantification for legacy data observations
  of {L}ipschitz functions.
\newblock \emph{ESAIM Math. Model. Numer. Anal.}, 47\penalty0 (6):\penalty0
  1657--1689, 2013.
\newblock \doi{10.1051/m2an/2013083}.

\bibitem[Tibshirani et~al.(2011)Tibshirani, Hoefling, and
  Tibshirani]{Tibshirani2011}
R.~J. Tibshirani, H.~Hoefling, and R.~Tibshirani.
\newblock Nearly-isotonic regression.
\newblock \emph{Technometrics}, 53\penalty0 (1):\penalty0 54--61, 2011.
\newblock \doi{10.1198/TECH.2010.10111}.

\end{thebibliography}

\end{document}